\documentclass[11pt,a4paper]{article}
\usepackage[margin=2.6cm]{geometry}
\usepackage{setspace}
\usepackage{amsmath}
\usepackage{amsfonts}
\usepackage{amssymb}
\usepackage{amsthm}
\usepackage{mathtools}
\usepackage{bm}

\usepackage{graphicx}
\usepackage{float}
\usepackage{subcaption}
\usepackage[export]{adjustbox}
\usepackage{tikz}
\usetikzlibrary{arrows,arrows.meta,shapes,patterns,positioning,calc}

\usepackage{booktabs}
\usepackage{tabularx}
\usepackage{array}

\usepackage{xcolor}
\usepackage[colorlinks=true,linkcolor=blue!50!black,
            citecolor=blue!50!black,urlcolor=blue!50!black]{hyperref}
\usepackage[capitalise,noabbrev]{cleveref}

\providecommand{\doi}[1]{DOI:~\href{https://doi.org/#1}{\nolinkurl{#1}}}

\theoremstyle{plain}
\newtheorem{theorem}{Theorem}
\newtheorem{proposition}{Proposition}
\newtheorem{lemma}{Lemma}
\newtheorem{corollary}{Corollary}
\theoremstyle{definition}
\newtheorem{definition}{Definition}
\newtheorem{example}{Example}
\theoremstyle{remark}
\newtheorem{remark}{Remark}

\newcommand{\R}{\mathbb{R}}

\newcommand{\bx}{\mathbf{x}}

\newcommand{\bu}{\mathbf{u}}
\newcommand{\diff}{\,\mathrm{d}}
\DeclareMathOperator*{\sgn}{sgn}
\newcommand{\fprod}{f_i^{\,\mathrm{prod}}}
\newcommand{\fwgtbc}{f_i^{\,\mathrm{wgt,bc}}}
\newcommand{\fwgt}{f_i^{\,\mathrm{wgt}}}
\newcommand{\fS}{f_i^{\,\mathrm{S}}}
\newcommand{\Acal}{\mathcal{A}}
\newcommand{\Rcal}{\mathcal{R}}
\newcommand{\half}{\tfrac{1}{2}}

\newcommand{\eqbox}[1]{\boxed{\;\;#1\;\;}}

\title{\bfseries A Bias-Corrected Weighted Logistic Model for Gene
       Regulatory Networks: Functional Equivalence with the
       Product-of-Logistics and Comparison with Weighted-Sum
       Formulations}

\author{%
  \textsc{Ismail Belgacem}\thanks{Corresponding author: \texttt{ismail.belgacem.81@gmail.com}.}
  \\[2pt]
  \small Independent Researcher,\\
  \small Mezaourou, Ghazaouet, Tlemcen, Algeria
}
\date{}

\begin{document}

\maketitle

\begin{abstract}
We introduce a \emph{bias-corrected weighted-logistic} (bcw)
formulation for ODE models of gene regulatory networks (GRNs). Each
gene's regulatory function is a single sigmoid
$\sigma(\lambda S_i(\bx) + \lambda b_i)$ of the signed weighted
regulator sum $S_i$, with a per-gene bias
$b_i = -\lambda^{-1}\log(2^{m_i}-1)$ depending only on the in-degree
$m_i$ and the steepness $\lambda$. We prove that this bias is the
unique single-sigmoid correction that recovers the critical-point
value $1/2^{m_i}$ of the product-of-logistics formulation and shares
its canonical equilibrium $x_i^{*} = \kappa_i/(\gamma_i\,2^{m_i})$
under self-consistency. The framework yields: an exact algebraic
discrepancy identity expressing $\fprod - \fwgtbc$ as a finite sum
over the proper non-empty subsets of the regulator index set; a
slope-ratio formula $2-2^{1-m_i}$ at the critical point; a curvature
mismatch showing the structural symmetry $\sigma''=0$ is preserved by
the product but broken by the bcw form for $m_i\ge 2$; a no-go result
ruling out global equivalence by any state-independent constant; and
a mean-field interpretation of the bias as a zero-order log-sum-exp
approximation. At a shared equilibrium the Jacobians are related by
the explicit row-wise transformation $J^{\mathrm{bcw}} = D(J^{\mathrm{prod}}
+\Gamma) - \Gamma$ with $D = \mathrm{diag}(2-2^{1-m_i})$, yielding a
stability dichotomy that classifies which spectral modes the bias
correction destabilises or further stabilises. In the contractive
regime $\gamma_i > \kappa_i\lambda m_i/4$ both formulations converge
globally and exponentially to the shared equilibrium at the explicit
rate $\alpha = \min_i(\gamma_i - \kappa_i\lambda m_i/4)$. A canonical
two-gene self-activating negative-feedback motif illustrates the
theory: the bcw formulation undergoes Hopf bifurcation at
$\lambda^{\mathrm{bcw}}_{\mathrm{Hopf}} = \tfrac{2}{3}\,
\lambda^{\mathrm{prod}}_{\mathrm{Hopf}}$ and sustains limit-cycle
oscillations around the shared equilibrium in a parameter regime
where the product-of-logistics remains linearly stable, with the
predicted eigenvalue amplification factor $D = 3/2$.
\end{abstract}

\medskip


\medskip

\bigskip

\section{Introduction}
\label{sec:introduction}

Sigmoidal regulatory functions are the workhorse of mathematical models
of gene regulatory networks (GRNs). For decades, the Hill family
$h^+(x,\theta,n) = x^n/(\theta^n+x^n)$ and
$h^-(x,\theta,n) = \theta^n/(\theta^n+x^n)$ has supplied the standard
parameterisation, encoding cooperativity through the Hill coefficient
$n$ and threshold sensitivity through the dissociation constant
$\theta$~\cite{santillan2008use,bintu2005transcriptional}. The
mechanistic origin of Hill functions in equilibrium binding theory
generalises the Michaelis--Menten kinetics that govern enzymatic chains
and, more broadly, the global-stability analyses developed
in~\cite{belgacem2013michaelis,belgacem2012michaelis} for full
reversible Michaelis--Menten reactions. Despite the long-standing
success of Hill functions, they develop derivative singularities at
the origin whenever $n$ is non-integer
($h^{+\prime}(x)\to\infty$ as $x\to 0^+$ for $0<n<1$, with
higher-order derivatives diverging on $n\in(k,k+1)$), lack
closed-form antiderivatives and inverses, and admit only a
logarithmically divergent small-$n$ expansion~\cite{santillan2008use,
hernandez2023corrected,abramowitz1965handbook}. These pathologies
compromise numerical integration in low-expression regimes,
ill-condition Fisher information matrices for parameter estimation, and
prevent reliable observability analysis in feedback control
applications.

A growing body of work has therefore proposed replacing Hill functions
by their logistic counterparts
$f^+(x,\theta,\lambda) = 1/(1+e^{-\lambda(x-\theta)})$ and
$f^-(x,\theta,\lambda) = 1/(1+e^{\lambda(x-\theta)})$, with the
parameter matching $\lambda = n/\theta$ preserving local input--output
sensitivity at the half-maximal concentration~\cite{belgacem2025exploring,
belgacem2026logistic}. The logistic functions are globally
$C^\infty$, possess closed-form derivatives with the self-referential
structure $\partial f^\pm/\partial x = \pm\lambda f^\pm(1-f^\pm)$,
admit elementary antiderivatives $\int f^+\,\diff x =
\lambda^{-1}\ln(1+e^{\lambda(x-\theta)})$, are inverted in closed form
through the logit transformation, and are uniformly Lipschitz with
constant $\lambda/4$. The logistic formulation also captures basal
expression naturally through $f^+(0,\theta,\lambda) = 1/(1+e^{\lambda
\theta})>0$, without ad hoc additive offsets.

\subsection{Two competing logistic-based GRN frameworks}
\label{subsec:competing}

Two distinct logistic-based GRN formulations have emerged in the
literature.

\paragraph*{Product-of-logistics models} treat each regulatory
interaction as an independent sigmoid factor and combine factors
multiplicatively to capture parallel (AND-like) regulation, with each
regulator carrying its own threshold $\theta_{i,m}$ and the directional
character ($+$ for activator, $-$ for repressor) encoded inside the
exponent. This formulation, developed
in~\cite{belgacem2025exploring,belgacem2026logistic} and related
works on stability and reduction of gene-expression
models~\cite{belgacem2018reduction,belgacem2014mathematical,
belgacem2014stability,belgacem2013stability,belgacem2013analysis},
is the smooth analytic counterpart of the logical step-function GRN
networks introduced by Glass and
Kauffman~\cite{glass1973logical}, whose high-dimensional periodicity,
chaos, and bifurcation structure on ring circuits have been further
analysed in~\cite{belgacem2025glass,farcot2019chaos}. It preserves a
transparent biological interpretation: each $\theta_{i,m}$ maps to a
measurable binding affinity $K_d$ or half-maximal effective
concentration $\mathrm{EC}_{50}$, and the value of the regulatory
function at the critical point ($x_{j(i,m)} = \theta_{i,m}$ for every
$m$) equals $(1/2)^{m_i}$, the natural Bernoulli-product probability.
The explicit form is introduced as \eqref{eq:f_prod_def} in
\Cref{subsec:product_baseline}, with an equivalent sign-encoded compact
form $\fprod = \prod_m\sigma(\lambda\,\epsilon_{i,m}(x_{j(i,m)}-\theta_{i,m}))$
recorded in \Cref{rem:prod_compact}.

\paragraph*{Weighted-sum models}
~\cite{samuilik2022mathematical,kozlovska2022models,sadyrbaev2021modelling,
kozlovska2023quasi,samuilik2022genetic,sadyrbaev2023coexistence,
somathilaka2023revealing,vohradsky2001neural}
aggregate all regulatory inputs into a single weighted sum passed
through one increasing logistic per gene:
\begin{equation}
\dot{x}_i = \kappa_i\,
   \frac{1}{1 + e^{-\mu_i\!\bigl(\sum_{j=1}^n w_{ij}^{\mathrm{S}}x_j
                                                  - \theta_i^{\mathrm{S}}\bigr)}}
   - \gamma_i x_i,
\qquad w_{ij}^{\mathrm{S}}\in\R,
\label{eq:samuilik_intro}
\end{equation}
with regulatory direction encoded by the sign of $w_{ij}^{\mathrm{S}}$.
This formulation reduces the number of nonlinear terms to one per
gene, an attractive property for bifurcation analysis, parameter
estimation, and feedback-control synthesis. However, the
negative-weight encoding creates structural pathologies. The midpoint
of the increasing sigmoid satisfies $\sum_j w_{ij}^{\mathrm{S}}x_j =
\theta_i^{\mathrm{S}}$; for a single repressor with $w^{\mathrm{S}}<0$,
$\theta^{\mathrm{S}}>0$, this gives a midpoint $x_c =
\theta^{\mathrm{S}}/w^{\mathrm{S}}<0$ outside the biologically
admissible domain $x\ge 0$, so the sigmoid never undergoes its
sigmoidal transition over physical concentrations
(\Cref{subsec:repression_pathology}). The shared threshold prescription
$\theta_i^{\mathrm{S}} = \tfrac{1}{2}\sum_j w_{ij}^{\mathrm{S}}$
commonly adopted
in~\cite{samuilik2022mathematical,kozlovska2022models} produces
$\theta_i^{\mathrm{S}} = 0$ for any AND-gate combining one activator
and one repressor, leaving no characteristic biological scale to
identify parameters against.

\subsection{The bridge: a bias-corrected weighted formulation}
\label{subsec:bridge}

This paper develops a third logistic-based formulation that occupies
the middle ground between~\eqref{eq:product_of_logistics_intro} and
\eqref{eq:samuilik_intro}: a single sigmoid per gene with a
\emph{regulator-specific} normalisation and a \emph{combinatorial} bias
term. It directly follows and complements two recent preprints by
Belgacem~\cite{belgacem2025exploring,belgacem2026logistic}, which form
a two-stage program of work. The first
preprint~\cite{belgacem2025exploring} introduced and rigorously
catalogued the analytical advantages of replacing Hill functions
$h^\pm(x,\theta,n)$ by logistic counterparts
$f^\pm(x,\theta,\lambda)$ under the matching $\lambda = n/\theta$:
infinite differentiability, uniform Lipschitz bound $|\partial
f^\pm/\partial x|\le\lambda/4$, elementary inverses through the logit
transformation, and a non-zero basal output $f^+(0,\theta,\lambda) =
1/(1+e^{\lambda\theta}) > 0$. It also established global existence,
smoothness, and uniform boundedness for the product-of-logistics GRN
model with the explicit Lipschitz constant
$L_F\le M = \max_i(\kappa_i\sum_j L_i^j + \gamma_i)$
\cite[Thm.~1]{belgacem2025exploring} and discussed the structural
shortcomings of weighted-sum formulations relative to the product
\cite[\S 8.3]{belgacem2025exploring}. The second
preprint~\cite{belgacem2026logistic} deepened the case for the
logistic framework by demonstrating that, for canonical
two-gene oscillators and bistable positive-autoregulation motifs,
non-zero basal output prevents the Hill-induced off-state trap: in a
biophysically grounded \emph{E.~coli} parameter regime the logistic
model escapes the basal state in $\sim\!44$~min while the Hill model
remains permanently trapped~\cite[\S 5]{belgacem2026logistic}.

The present paper takes the product-of-logistics framework developed
in~\cite{belgacem2025exploring,belgacem2026logistic} as its analytic
baseline and asks: \emph{is there a single-sigmoid formulation that
retains the product's biological interpretability at the critical
point while collapsing the analytical complexity to one sigmoid per
gene}? The answer is the bias-corrected weighted-logistic formulation
below, which inherits the global mathematical guarantees of the
product framework (\Cref{thm:well_posedness}, in the form of
\cite[Thm.~1]{belgacem2025exploring}) and the strictly-positive basal
output that prevents expression shutdown (\Cref{cor:no_shutdown},
echoing the basal-escape mechanism of~\cite{belgacem2026logistic})
while collapsing the per-gene structure to one sigmoid.

We start from the unified weighted form
\begin{equation}
\dot{x}_i = \kappa_i\,
   \frac{1}{1 + \exp\!\bigl(-\lambda \sum_{j=1}^n
                                w_{ij}(x_j-\theta_{ij})\bigr)}
   - \gamma_i x_i,
\qquad w_{ij} \in \R\;\;\text{(sign-encoded direction)},\;\theta_{ij}>0,
\label{eq:weighted_intro}
\end{equation}
which preserves regulator-specific thresholds $\theta_{ij}$ (in
contrast to~\eqref{eq:samuilik_intro}, where the threshold $\theta_i$
is shared across regulators). The price for retaining a single sigmoid
per gene is a loss of agreement with the product-of-logistics at the
critical point: \eqref{eq:weighted_intro} attains the value $1/2$
when all regulators sit at their thresholds, while
\eqref{eq:product_of_logistics_intro} attains $(1/2)^{m_i}$, where
$m_i = |\{j: w_{ij}\ne 0\}|$ counts the regulators of gene $i$.

We restore agreement at this critical point---and at two other
biologically meaningful reference configurations---by introducing the
combinatorial bias
\[
b_i = -\frac{1}{\lambda}\,\log(2^{m_i}-1),
\]
yielding the bias-corrected weighted-logistic regulatory function
$\fwgtbc(\bx) = 1/(1+(2^{m_i}-1)e^{-\lambda S_i})$ with
$S_i = \sum_j w_{ij}(x_j-\theta_{ij})$.
The bias depends only on the regulator count $m_i$ and the steepness
$\lambda$, not on the biological parameters $(\kappa_i,\gamma_i,
\theta_{ij},w_{ij})$ of the network; the dimensionless logit shift
$\lambda b_i = -\log(2^{m_i}-1)$ depends on $m_i$ alone. The bias
forces $\fwgtbc=\fprod$ at three distinguished configurations.

\subsection{Contributions}
\label{subsec:contributions}

The contributions of this paper are:
\begin{enumerate}\itemsep2pt
\item A precise statement of the bias-corrected weighted formulation
      (\Cref{def:bcw_model}) and of its embedding of the increasing
      and decreasing logistic functions through additive sign
      encoding in the exponential argument
      (\Cref{prop:embedding}). The specific bias value
      $b_i = -\log(2^{m_i}-1)/\lambda$ is characterised as the unique
      single-sigmoid correction that recovers the product's
      critical-point value $1/2^{m_i}$ and shares its canonical
      equilibrium under self-consistency
      (\Cref{prop:bias_canonicity}); its dependence on the
      \emph{combinatorial} regulator count $m_i$ alone among the
      biological parameters $(\kappa_i,\gamma_i,\theta_{ij},w_{ij})$
      is the structural origin of the bcw formulation.
\item A global well-posedness theorem (\Cref{thm:well_posedness})
      for the bias-corrected weighted ODE system, with explicit
      Lipschitz bound $L_F\le M = \max_i(\kappa_i\lambda m_i/4 +
      \gamma_i)$ in the form of~\cite[Thm.~1]{belgacem2025exploring},
      a positively invariant hyper-rectangle $\prod_i[0,\kappa_i/\gamma_i]$,
      and a strictly-positive basal-output corollary
      (\Cref{cor:no_shutdown}) echoing the no-expression-shutdown
      property of~\cite{belgacem2026logistic}.
\item A three-point matching theorem (\Cref{thm:three_point}):
      $\fwgtbc=\fprod$ at the critical point, in the low-steepness
      limit, and at the saturation limits.
\item An exact algebraic identity for the discrepancy
      (\Cref{thm:algebraic_identity}) expressing
      $\fprod-\fwgtbc$ as a single rational function whose numerator
      is a sum over proper non-empty subsets of the regulator index
      set. The identity vanishes at the three reference configurations
      and provides exact bounds on the discrepancy at intermediate
      states. Sharp leading-order asymptotic decay rates of the
      discrepancy in both saturation limits are obtained as a
      consequence (\Cref{prop:asymptotic_rates}).
\item A closed-form slope-ratio formula (\Cref{prop:slope_ratio}):
      $|\partial \fwgtbc/\partial x_l|/|\partial\fprod/\partial x_l|
      = 2-2^{1-m_i}$ at the critical point, valid for any regulator
      $l$. The ratio is monotonically increasing in $m_i$ and
      bounded above by $2$.
\item A curvature-mismatch theorem (\Cref{prop:curvature_mismatch}):
      $\partial^2\fprod/\partial x_l^2 = 0$ at the critical point for
      every $m_i\ge 1$, while $\partial^2\fwgtbc/\partial x_l^2$
      vanishes only when $m_i=1$. The structural symmetry
      $\sigma''(\theta) = 0$ that determines the supercriticality of
      pitchfork and Hopf bifurcations in the product formulation is
      therefore broken by the bias correction for $m_i\ge 2$.
\item A no-go result (\Cref{prop:no_global_eq}) ruling out any
      state-independent shift $b\in\R$ that would yield global
      equivalence between weighted and product formulations for
      $m_i\ge 2$; a refined mean-field interpretation
      (\Cref{prop:mean_field}) identifying the bias as the
      zero-order constant-in-state approximation of the log-sum-exp
      correction; and a structural-limit result
      (\Cref{prop:affine_limit}) showing that no affine
      single-sigmoid family can simultaneously match the product's
      value, slope, \emph{and} curvature at the critical point.
\item An equilibrium-correspondence theorem
      (\Cref{thm:shared_equilibrium}) giving a parameter
      self-consistency condition under which the product-of-logistics
      and bias-corrected weighted ODE systems share a canonical
      critical-point equilibrium with regulatory output $1/2^{m_i}$,
      a global-uniqueness corollary
      (\Cref{cor:global_unique}) showing that in the contractive
      regime $\gamma_i>\kappa_i\lambda m_i/4$ this shared equilibrium
      is the unique globally attracting one for both systems with
      explicit exponential convergence rate $\alpha = \min_i
      (\gamma_i - \kappa_i\lambda m_i/4)$, a quantitative robustness
      result (\Cref{prop:robustness_eq}) for perturbations of the
      self-consistent thresholds, an explicit Jacobian transformation
      $J^{\rm bcw} = D(J^{\rm prod}+\Gamma)-\Gamma$
      (\Cref{prop:linearisation}) relating the linearised dynamics
      around the shared equilibrium, and a stability dichotomy
      (\Cref{cor:stability_dichotomy}) classifying which spectral
      modes are stabilised or destabilised under the bias correction.
\item A structural comparison with~\eqref{eq:samuilik_intro}, including
      the formal equivalence between fixed-weight and real-weight
      product-of-logistics under parameter rescaling
      (\Cref{prop:weight_rescaling}), and the demonstration that the
      Samuilik prescription $\theta_i^{\mathrm{S}} = \tfrac{1}{2}\sum_j
      w_{ij}^{\mathrm{S}}$ reduces, for canonical mixed regulation, to
      an inflection condition without characteristic biological scale
      (\Cref{subsec:mixed}).
\item A numerical illustration on a canonical two-gene
      negative-feedback oscillator with self-activation, in which
      the bias correction shifts the Hopf bifurcation threshold
      \emph{down} by the predicted factor $D^{-1} = 2/3$
      ($\lambda^{\rm bcw}_{\rm Hopf} = (2/3)\,\lambda^{\rm prod}_{\rm Hopf}$,
      with $D = 2-2^{1-m_i} = 3/2$ for $m_i = 2$ the eigenvalue
      amplification factor of \Cref{prop:linearisation}), and the
      bias-corrected formulation oscillates around the shared
      critical-point equilibrium while the product-of-logistics
      remains linearly stable.
\end{enumerate}

\subsection{Organisation}
\label{subsec:organisation}

\Cref{sec:formulation} introduces the logistic activation/repression
building blocks (\Cref{subsec:logistic_blocks}), recalls the
product-of-logistics framework obtained from Boolean regulatory rules
by a recursive De~Morgan map (\Cref{subsec:product_baseline}), and
formalises the Samuilik weighted-sum class
(\Cref{subsec:samuilik_def}), the unified weighted-logistic
formulation without bias (\Cref{subsec:weighted_def}), and the
bias-corrected weighted-logistic model that is the subject of the
analysis (\Cref{subsec:bcw_def}), with global well-posedness
established in \Cref{subsec:well_posedness}.
\Cref{sec:samuilik_comparison} compares the bias-corrected weighted
formulation with the Samuilik weighted-sum model, treating the
threshold structure, the repression pathology, and mixed
activation--repression.
\Cref{sec:equivalence} contains the principal new mathematical
results: the three-point matching theorem
(\Cref{thm:three_point}), the algebraic discrepancy identity
(\Cref{thm:algebraic_identity}), the slope-ratio formula
(\Cref{prop:slope_ratio}), the curvature analysis
(\Cref{prop:curvature_mismatch}), the no-go and mean-field results
(\Cref{prop:no_global_eq,prop:mean_field}), and the structural-limit
result (\Cref{prop:affine_limit}).
\Cref{sec:weight_rescaling} establishes the parameter-rescaling
equivalence between fixed-weight and real-weight product-of-logistics.
\Cref{sec:equilibrium} develops the dynamical correspondence between
the two formulations at shared critical-point equilibria, including
the linearisation comparison (\Cref{prop:linearisation}).
\Cref{sec:application} illustrates the equilibrium correspondence,
the eigenvalue map, and the stability dichotomy on a canonical
two-gene negative-feedback oscillator with self-activation.
\Cref{sec:discussion} discusses practical guidance and outlooks.
\Cref{sec:conclusion} concludes.

\section{Model Formulation}
\label{sec:formulation}

We formalise four sigmoidal regulatory models for gene networks. After
introducing the elementary logistic building blocks
(\Cref{subsec:logistic_blocks}), we recall the product-of-logistics
framework obtained by mapping Boolean regulatory rules through a
recursive De~Morgan map (\Cref{subsec:product_baseline}), then
introduce the Samuilik weighted-sum class
(\Cref{subsec:samuilik_def}), the unified weighted-logistic
formulation without bias (\Cref{subsec:weighted_def}), and finally the
bias-corrected weighted-logistic model that is the subject of the
analysis (\Cref{subsec:bcw_def}). \Cref{subsec:well_posedness}
establishes global well-posedness of the bias-corrected weighted ODE.

Throughout, $\bx(t)=(x_1(t),\ldots,x_n(t))^\top\in\R^n$ denotes the
vector of gene-expression levels in a network of $n$ genes. Each
gene's dynamics obey a balance between sigmoidally regulated synthesis
and linear degradation:
\begin{equation}
\dot{x}_i \;=\; \kappa_i\,f_i(\bx) - \gamma_i\,x_i,
\qquad i=1,\ldots,n,
\label{eq:general_GRN}
\end{equation}
where $\kappa_i>0$ is the maximal synthesis rate, $\gamma_i>0$ the
degradation rate, and $f_i:\R^n\to(0,1)$ the regulatory function. We
write $\Acal_i$ and $\Rcal_i$ for the index sets of activators and
repressors of gene $i$, and $m_i = |\Acal_i|+|\Rcal_i|$ for the total
regulator count.

\subsection{Logistic activation and repression functions}
\label{subsec:logistic_blocks}

The elementary building blocks of every sigmoidal GRN formulation
considered here are the \emph{increasing} logistic function modelling
gene activation,
\begin{equation}
f^+(x, \theta, \lambda) \;=\; \frac{1}{1 + e^{-\lambda (x - \theta)}},
\label{eq:logistic_activation}
\end{equation}
and the \emph{decreasing} logistic function modelling gene repression,
\begin{equation}
f^-(x, \theta, \lambda) \;=\; \frac{1}{1 + e^{\lambda (x - \theta)}}
\;=\; \frac{1}{1 + e^{-\lambda (\theta - x)}}
\;=\; 1 - f^+(x, \theta, \lambda),
\label{eq:logistic_repression}
\end{equation}
both with steepness $\lambda>0$ and regulator-specific threshold
$\theta>0$. The analytical advantages of $f^\pm$ over Hill functions
$h^\pm(x,\theta,n) = x^n/(\theta^n+x^n)$ (and the parameter matching
$\lambda = n/\theta$ that preserves slope at the half-maximum), namely
global $C^\infty$-smoothness, uniform Lipschitz bound
$|\partial f^\pm/\partial x|\le \lambda/4$, closed-form derivatives
$\partial f^\pm/\partial x = \pm\lambda f^\pm(1-f^\pm)$, elementary
antiderivatives, logit inversions, and strictly positive basal output
$f^+(0,\theta,\lambda) = 1/(1+e^{\lambda\theta})>0$, are catalogued in
detail in~\cite{belgacem2025exploring,belgacem2026logistic} and
underlie everything that follows.

\subsection{General multi-gene regulatory networks via Boolean-to-logistic recursion}
\label{subsec:product_baseline}

Each gene's regulatory logic can be encoded as a Boolean function of
its activators and repressors. The continuous-state \emph{product-of-logistics}
counterpart of such a rule is obtained by a recursive map $\Phi$
from Boolean expressions to functions $\R^n\to(0,1)$, defined on the
following primitive constructions
(see~\cite{belgacem2025exploring,belgacem2026logistic} for the
underlying derivation and a complete catalogue of structural
properties):

\begin{itemize}\itemsep1pt
\item \textbf{Positive literal (activation):} $\Phi(x_j) = f^+(x_j,\theta_{ij},\lambda)$.
\item \textbf{Negation (repression):} $\Phi(\lnot x_k) =
       f^-(x_k,\theta_{ik},\lambda) = 1 - f^+(x_k,\theta_{ik},\lambda)$.
\item \textbf{Conjunction (AND, cooperative co-regulation):}
       $\Phi\bigl(\bigwedge_{l=1}^{k} L_l\bigr) = \prod_{l=1}^{k}\Phi(L_l)$.
\item \textbf{Disjunction (OR, independent activation), via De~Morgan:}
\begin{equation}
\Phi\!\Bigl(\bigvee_{l=1}^{m} C_l(\bx)\Bigr)
\;=\;
1 - \prod_{l=1}^{m}\bigl(1 - \Phi(C_l(\bx))\bigr),
\label{eq:demorgan}
\end{equation}
       which follows from $\lnot(\lnot C_1\wedge\cdots\wedge\lnot C_m)$
       applied recursively and coincides with the probability that at
       least one of $m$ independent events occurs.
\end{itemize}
The recursion is consistent (Boolean values map to the interval
endpoints, AND/NOT/OR satisfy their classical truth tables when
inputs saturate, and $\Phi\in(0,1)$ on every state) and produces a
smooth surrogate of the underlying Boolean rule. The De~Morgan
product~\eqref{eq:demorgan} guarantees $\Phi\in[0,1]$ for any number
of independent regulatory pathways, in contrast to additive
translations of Boolean OR which can reach $m\gg 1$ when several
clauses fire simultaneously, inflating the production rate and
breaking the biological bound $x_i\le\kappa_i/\gamma_i$.

For the canonical case of $m_i$ regulators in a single AND clause
($|\Acal_i|$ activators and $|\Rcal_i|$ repressors, all required
simultaneously), the recursion specialises to the
\emph{product-of-logistics regulatory function}
\begin{equation}
\fprod(\bx) \;=\;
   \prod_{j\in\Acal_i}\!\frac{1}{1+e^{-\lambda(x_j-\theta_{ij})}}
   \;\cdot\;
   \prod_{k\in\Rcal_i}\!\frac{1}{1+e^{-\lambda(\theta_{ik}-x_k)}},
\label{eq:f_prod_def}
\end{equation}
which serves as the analytical baseline of this paper. The
mathematical foundations of \eqref{eq:f_prod_def}---global existence,
$C^\infty$-smoothness, uniform boundedness, an explicit Lipschitz
constant, monotonicity, diagonal-dominance criteria for global
stability, and closed-form bifurcation theorems on canonical
motifs---are established in detail
in~\cite{belgacem2025exploring,belgacem2026logistic} and we treat them
as established background.

\begin{remark}[Critical point of the product]
\label{rem:critical_point}
The \emph{critical point} of gene $i$ is the configuration
$\bx^{(\mathrm{cp})}_i$ at which every regulator of $i$ sits at its
threshold: $x_j = \theta_{ij}$ for all $j\in\Acal_i\cup\Rcal_i$. At
this point each factor of \eqref{eq:f_prod_def} equals $\half$, so
$\fprod(\bx^{(\mathrm{cp})}_i) = (1/2)^{m_i}$.
\end{remark}

\begin{remark}[Equivalent compact form with sign-encoded directions]
\label{rem:prod_compact}
Enumerating the regulators of gene $i$ by $m\in\{1,\ldots,m_i\}$ via
an index map $j(i,m)$ and recording the regulatory direction by
$\epsilon_{i,m}\in\{+1,-1\}$ (with $\epsilon_{i,m}=+1$ for $j(i,m)\in\Acal_i$
and $\epsilon_{i,m}=-1$ for $j(i,m)\in\Rcal_i$), the product
formulation \eqref{eq:f_prod_def} admits the compact sign-encoded form
\begin{equation}
\dot{x}_i \;=\; \kappa_i\,\fprod(\bx) - \gamma_i x_i,
\qquad
\fprod(\bx) \;=\; \prod_{m=1}^{m_i}
\frac{1}{1+\exp\!\bigl(-\lambda\,\epsilon_{i,m}\,(x_{j(i,m)}-\theta_{i,m})\bigr)},
\label{eq:product_of_logistics_intro}
\end{equation}
where $\theta_{i,m}\equiv\theta_{i,j(i,m)}$ is the threshold of the
$m$-th regulator. Both representations describe identical dynamics and
reduce to one another once $\epsilon_{i,m}$ is read as the encoding of
the dichotomy $\Acal_i\sqcup\Rcal_i$. The construction extends to
regulator-specific steepness $\lambda_{i,m}$ at the cost of a more
elaborate bias formula in the bias-corrected weighted formulation; see
\Cref{rem:nonuniform_lambda}.
\end{remark}

\subsection{Weighted-sum models}
\label{subsec:samuilik_def}

The product structure of \eqref{eq:f_prod_def} requires $m_i$ logistic
evaluations per gene; for high-dimensional networks an alternative
tradition aggregates regulatory inputs into a single weighted sum
passed through one increasing logistic per
gene~\cite{samuilik2022mathematical,kozlovska2022models,
sadyrbaev2021modelling,kozlovska2023quasi,samuilik2022genetic,
sadyrbaev2023coexistence,somathilaka2023revealing,vohradsky2001neural}:
\begin{equation}
\dot{x}_i \;=\; \kappa_i\,\fS(\bx) - \gamma_i x_i,
\qquad
\fS(\bx) \;=\;
\frac{1}{1+\exp\!\bigl(-\mu_i\bigl(\sum_{j=1}^n w_{ij}^{\mathrm{S}}x_j -
                              \theta_i^{\mathrm{S}}\bigr)\bigr)},
\qquad w_{ij}^{\mathrm{S}}\in\R,
\label{eq:samuilik_def}
\end{equation}
with gene-specific steepness $\mu_i$, real-valued signed weights
$w_{ij}^{\mathrm{S}}$ (positive for activation, negative for
repression), and a \emph{single} shared threshold $\theta_i^{\mathrm{S}}$
per gene. The canonical threshold prescription
in~\cite{samuilik2022mathematical,kozlovska2022models} is
$\theta_i^{\mathrm{S}} = \tfrac{1}{2}\sum_j w_{ij}^{\mathrm{S}}$.
Formulation \eqref{eq:samuilik_def} compresses $m_i$ nonlinear terms
into one, an attractive feature for control synthesis and parameter
estimation, but the consequences of using a single shared threshold and
signed weights inside an \emph{increasing} sigmoid (a critical-point
pathology for repressors and a threshold-scaling pathology) were
identified in~\cite[\S 8.3]{belgacem2025exploring}; we revisit and
quantify them in \Cref{sec:samuilik_comparison}.

\subsection{The unified weighted-logistic formulation (without bias)}
\label{subsec:weighted_def}

A natural single-sigmoid alternative to \eqref{eq:samuilik_def} that
\emph{retains} regulator-specific thresholds $\theta_{ij}$ is the
weighted-logistic formulation:

\begin{definition}[Unified weighted-logistic GRN model]
\label{def:weighted_model}
The \emph{unified weighted-logistic} GRN model is
\begin{equation}
\dot{x}_i \;=\; \kappa_i\,\fwgt(\bx) - \gamma_i x_i,
\qquad
\fwgt(\bx) \;=\;
\frac{1}{1 + \exp\!\bigl(-\lambda S_i(\bx)\bigr)},
\quad
S_i(\bx) \;=\; \sum_{j=1}^n w_{ij}(x_j-\theta_{ij}),
\label{eq:weighted_def}
\end{equation}
with parameters $\kappa_i,\gamma_i,\lambda > 0$, regulator-specific
thresholds $\theta_{ij}>0$, and signed weights $w_{ij}\in\R$ encoding
the direction of interaction by sign:
\[
w_{ij} > 0 \;\Longleftrightarrow\; j\in\Acal_i \;\text{(activator)},
\qquad
w_{ij} < 0 \;\Longleftrightarrow\; j\in\Rcal_i \;\text{(repressor)},
\qquad
w_{ij} = 0 \;\Longleftrightarrow\; \text{no interaction}.
\]
The total regulator count is $m_i = \bigl|\{j: w_{ij}\ne 0\}\bigr|$.
\end{definition}

\begin{proposition}[Embedding of activation and repression as signed weights]
\label{prop:embedding}
The affine displacement appearing in $S_i$ admits the split
\begin{equation}
w_{ij}(x_j-\theta_{ij}) \;=\;
\begin{cases}
+|w_{ij}|\,(x_j-\theta_{ij}), & w_{ij}>0,\\[1pt]
-|w_{ij}|\,(x_j-\theta_{ij}) \;=\; |w_{ij}|\,(\theta_{ij}-x_j), & w_{ij}<0,
\end{cases}
\label{eq:S_split}
\end{equation}
so activation and repression are encoded by the sign of $w_{ij}$
within a single increasing sigmoid in \eqref{eq:weighted_def}, with
$|w_{ij}|$ playing the role of a regulator-specific gain.
\end{proposition}

\begin{proof}
Direct factoring of $w_{ij}(x_j-\theta_{ij})$. The interpretation in
terms of $f^\pm$ is immediate from \eqref{eq:logistic_activation}--%
\eqref{eq:logistic_repression}.
\end{proof}

\Cref{prop:embedding} establishes that activation and repression are
encoded \emph{within the exponential argument} of a single sigmoid in
the weighted formulation, in contrast with the product formulation
\eqref{eq:f_prod_def} which encodes them \emph{multiplicatively}
through distinct functional forms. This unification is the source of
the weighted formulation's analytical compactness but also the source
of the agreement failure that motivates the bias correction.

\begin{remark}[Magnitude rescaling]
\label{rem:magnitude_rescaling}
Under the embedding \eqref{eq:S_split}, the magnitude $|w_{ij}|$
multiplies $\lambda$ and rescales each $\theta_{ij}$. By the
parameter-rescaling identity established in
\Cref{prop:weight_rescaling} below, the formulation
\eqref{eq:weighted_def} with arbitrary real-valued $w_{ij}$ is
dynamically equivalent to one with unit magnitudes $|w_{ij}|=1$ and
rescaled $(\lambda,\theta_{ij})$. The analysis below is conducted
under the convention $|w_{ij}|=1$ without loss of generality,
restoring real magnitudes only where explicitly relevant.
\end{remark}

\subsection{Critical-point mismatch and the bias correction}
\label{subsec:bcw_def}

We now record the critical-point value mismatch between the weighted
and product formulations, and introduce the bias-corrected weighted
formulation that resolves it.

\begin{lemma}[Critical-point values]
\label{lem:critical_values}
At the critical point $\bx_i^{(\mathrm{cp})}$ of gene $i$:
\begin{equation}
\fprod(\bx_i^{(\mathrm{cp})}) = (1/2)^{m_i},
\qquad
\fwgt(\bx_i^{(\mathrm{cp})}) = 1/2.
\end{equation}
The two values agree iff $m_i = 1$.
\end{lemma}

\begin{proof}
The product value is computed in \Cref{rem:critical_point}. For the
weighted form, $S_i(\bx_i^{(\mathrm{cp})}) = \sum_j w_{ij}\cdot 0 = 0$,
so $\fwgt = 1/(1+e^0) = 1/2$. Equality $(1/2)^{m_i} = 1/2$ holds iff
$m_i=1$.
\end{proof}

\begin{lemma}[Low-steepness and saturation values]
\label{lem:limits}
With $u_j = w_{ij}\lambda(x_j-\theta_{ij})$ for $j\in[m_i]$ and
$S_i = \sum_j w_{ij}(x_j-\theta_{ij})$:
\begin{itemize}\itemsep1pt
\item[\textnormal{(a)}] (\emph{Low-steepness limit}.) On every compact
$K\subset\R^n$, as $\lambda\to 0^+$,
\[
\fprod(\bx)\to (1/2)^{m_i},\qquad \fwgt(\bx)\to 1/2,
\]
uniformly in $\bx\in K$.
\item[\textnormal{(b)}] (\emph{Joint positive saturation}.) Whenever
$u_j\to +\infty$ for every $j\in[m_i]$ (which forces $\lambda S_i\to
+\infty$):
\[
\fprod(\bx)\to 1, \qquad \fwgt(\bx)\to 1.
\]
\item[\textnormal{(c)}] (\emph{Joint negative saturation}.) Whenever
$u_j\to -\infty$ for every $j\in[m_i]$ (which forces $\lambda S_i\to
-\infty$):
\[
\fprod(\bx)\to 0, \qquad \fwgt(\bx)\to 0.
\]
\end{itemize}
\end{lemma}

\begin{proof}
\textbf{(a)} For $\lambda\to 0^+$ at fixed $\bx\in K$, every factor
$\sigma(u_j) = \sigma(\lambda w_{ij}(x_j-\theta_{ij}))\to \sigma(0) =
1/2$, and the sigmoid Lipschitz bound $|\sigma'|\le 1/4$ gives
uniformity on $K$. Hence $\fprod\to (1/2)^{m_i}$. The weighted
argument satisfies $|\lambda S_i|\le \lambda\sup_{\bx\in K}|S_i(\bx)|
\to 0$, so $\fwgt=\sigma(\lambda S_i)\to 1/2$ uniformly on $K$.
\textbf{(b)} If $u_j\to +\infty$ for every $j$, then
$\sigma(u_j)\to 1$ for every $j$ and the product
$\prod_j\sigma(u_j)\to 1$; also $\lambda S_i = \sum_j u_j\to +\infty$
forces $\fwgt = \sigma(\lambda S_i)\to 1$.
\textbf{(c)} Symmetric: if $u_j\to -\infty$ for every $j$, then
$\sigma(u_j)\to 0$, $\fprod\to 0$, and $\lambda S_i\to -\infty$
forces $\fwgt\to 0$.
\end{proof}

\begin{remark}[On the joint-saturation hypothesis]
\Cref{lem:limits}(b)--(c) state \emph{joint} saturation limits; the
condition $\lambda S_i\to\pm\infty$ alone does not determine $\fprod$
without the joint condition on the $u_j$'s, because $\lambda S_i$
can be driven to $+\infty$ by a few large positive $u_j$ even while
other $u_j$'s remain negative, in which case $\fprod\to 0$ (the
single saturated repressor dominates the product).
\end{remark}

The critical-point mismatch $(1/2)^{m_i} \ne 1/2$ for $m_i\ge 2$ is
\emph{combinatorial} in origin: it is determined entirely by the
regulator count. This suggests resolving it by a corresponding
combinatorial shift of the weighted formulation's threshold, which is
the content of the next definition.

\begin{definition}[Bias-corrected weighted-logistic GRN model]
\label{def:bcw_model}
The \emph{bias-corrected weighted-logistic} GRN model is
\begin{equation}
\dot{x}_i \;=\; \kappa_i\,\fwgtbc(\bx) - \gamma_i x_i,
\qquad
\fwgtbc(\bx) \;=\;
\frac{1}{1 + \exp\!\bigl(-\lambda(S_i(\bx) + b_i)\bigr)},
\quad
b_i \;=\; -\frac{1}{\lambda}\,\log(2^{m_i}-1),
\label{eq:bcw_model}
\end{equation}
with $S_i$ and $m_i$ as in \Cref{def:weighted_model}. Equivalently,
\begin{equation}
\fwgtbc(\bx) \;=\;
\frac{1}{1 + (2^{m_i}-1)\,\exp\!\bigl(-\lambda S_i(\bx)\bigr)}
\;=\;
\frac{e^{\lambda S_i(\bx)}}{(2^{m_i}-1) + e^{\lambda S_i(\bx)}}.
\label{eq:bcw_compact}
\end{equation}
\end{definition}

The bias $b_i$ depends only on the regulator count $m_i$ and the
steepness $\lambda$. It does not depend on the parameters
$(\kappa_i,\gamma_i,\theta_{ij},w_{ij})$ and is therefore determined
\emph{entirely by the combinatorial structure of the network}.
\Cref{def:bcw_model} reduces to the unified weighted form
\eqref{eq:weighted_def} when $m_i=1$ (since then $b_i=0$).

The specific value $b_i = -\log(2^{m_i}-1)/\lambda$ is not an arbitrary
choice but the unique constant for which the resulting single-sigmoid
model agrees with the product-of-logistics at the critical point and
shares its canonical equilibrium under self-consistency. We make this
precise in the next proposition.

\begin{proposition}[Canonicity of the bias correction]
\label{prop:bias_canonicity}
Fix a network topology $(w_{ij},\theta_{ij})_{i,j}$ with regulator
counts $m_i\ge 1$ and steepness $\lambda>0$. Consider the
one-parameter family of single-sigmoid regulatory models
\begin{equation}
\tilde{f}_i(\bx;\,c_i) \;:=\;
\sigma\!\bigl(\lambda\,S_i(\bx) + c_i\bigr),
\qquad c_i\in\mathbb{R},
\label{eq:generic_corrected}
\end{equation}
where $S_i$ is the signed regulator sum of \Cref{def:weighted_model}.
Then:
\begin{enumerate}\itemsep1pt
\item[\textnormal{(i)}] (\textbf{Critical-point matching uniqueness.}) The constant
$c_i$ is unique with the property that $\tilde{f}_i$ agrees with the
product-of-logistics at the critical point
$\bx_i^{(\mathrm{cp})}$ of gene $i$:
\begin{equation}
\tilde{f}_i(\bx_i^{(\mathrm{cp})};\,c_i)
\;=\; \fprod(\bx_i^{(\mathrm{cp})})
\;=\; \frac{1}{2^{m_i}}
\qquad\Longleftrightarrow\qquad
c_i \;=\; -\log(2^{m_i}-1).
\label{eq:bias_unique_cp}
\end{equation}
\item[\textnormal{(ii)}] (\textbf{Equilibrium-sharing uniqueness.}) Suppose the network
parameters satisfy the self-consistency conditions $\theta_{ij} =
x_j^{*}$ of \Cref{thm:shared_equilibrium}. Then $c_i = -\log(2^{m_i}
-1)$ is the unique constant for which the ODE
$\dot{x}_i = \kappa_i\tilde{f}_i(\bx;\,c_i) - \gamma_i x_i$ admits the
canonical critical-point equilibrium
\[
x_i^{*} \;=\; \frac{\kappa_i}{\gamma_i\,2^{m_i}},
\qquad i = 1,\ldots,n,
\]
shared with the product-of-logistics.
\item[\textnormal{(iii)}] (\textbf{Combinatorial dependence.}) The canonical bias
$c_i = -\log(2^{m_i}-1)$ depends only on the regulator count $m_i$;
it is independent of $\lambda$ as a logit value and independent of
$(\kappa_i,\gamma_i,\theta_{ij},w_{ij})$ as a parameter of the
network. In particular, an additive constant correction that adapts
only to the local in-degree of each gene is sufficient to restore
critical-point agreement.
\end{enumerate}
\end{proposition}

\begin{proof}
\textbf{(i)} At the critical point $S_i(\bx_i^{(\mathrm{cp})}) = 0$
(\Cref{rem:critical_point}), so
$\tilde{f}_i(\bx_i^{(\mathrm{cp})};\,c_i) = \sigma(c_i)$. Equating this
to $1/2^{m_i}$ and applying the logit transformation gives
\[
c_i \;=\; \log\!\Bigl(\frac{1/2^{m_i}}{1-1/2^{m_i}}\Bigr)
\;=\; \log\!\Bigl(\frac{1}{2^{m_i}-1}\Bigr)
\;=\; -\log(2^{m_i}-1),
\]
and the strict monotonicity of $\sigma$ on $\mathbb{R}$ makes the
solution unique.
\textbf{(ii)} Under self-consistency every regulator of gene $i$ sits
at its threshold at $\bx^{*}$, so $S_i(\bx^{*})=0$. The equilibrium
condition $\dot{x}_i = 0$ at $\bx^{*}$ reads $\kappa_i\sigma(c_i) =
\gamma_i x_i^{*}$, hence $\sigma(c_i) = \gamma_i x_i^{*}/\kappa_i$.
Requiring $x_i^{*} = \kappa_i/(\gamma_i\,2^{m_i})$ forces $\sigma(c_i)
= 1/2^{m_i}$, and (i) gives the unique solution $c_i = -\log(2^{m_i}
-1)$.
\textbf{(iii)} Both formulas $c_i = -\log(2^{m_i}-1)$ and $b_i =
c_i/\lambda$ contain $m_i$ alone among the network parameters; the
factor $1/\lambda$ in $b_i$ converts a dimensionless logit shift into
a state-space threshold offset, but the dimensionless shift $c_i$
itself depends only on $m_i$.
\end{proof}

\Cref{prop:bias_canonicity} establishes that
\emph{any} single-sigmoid corrected formulation that aligns its
critical-point value with the product-of-logistics must use precisely
the bias $b_i = -\log(2^{m_i}-1)/\lambda$. There is no other constant
correction that achieves either critical-point matching or
equilibrium sharing, so the bcw formulation is the unique
single-sigmoid sibling of the product in this sense. The
canonicity of \Cref{prop:bias_canonicity} is thus a \emph{local
zeroth-order} statement (exact value at the critical point, exact
equilibrium at the canonical configuration under self-consistency).
Higher-order obstructions remain: by the slope-ratio formula
\eqref{eq:slope_ratio} of \Cref{prop:slope_ratio}, the leading-order
discrepancy in the direction of regulator $l$ satisfies
\[
\fprod(\bx) - \fwgtbc(\bx)
\;=\; -(1 - 2^{1-m_i})\,\partial_{x_l}\fprod\big|_{\rm cp}\,
        (x_l - \theta_{il}) \;+\; O(\|\bx-\bx_i^{(\mathrm{cp})}\|^2),
\]
which is linear in the displacement with explicit coefficient
$(1-2^{1-m_i})\,\partial_{x_l}\fprod$ in absolute value; the
curvature mismatch \eqref{eq:fwgt_curv} of
\Cref{prop:curvature_mismatch} is the next-order obstruction, and the
no-go result of \Cref{prop:no_global_eq} below shows that no constant
correction can extend the local matching to a global one.
\Cref{prop:bias_canonicity} should therefore be read as
characterising the unique canonical \emph{constant} bias compatible
with the product's combinatorial critical-point structure, not as a
claim of full local equivalence.

\subsection{Global well-posedness of the bias-corrected weighted ODE}
\label{subsec:well_posedness}

We record that the bias-corrected weighted-logistic system inherits
the structural guarantees that motivate the product-of-logistics
framework of \cite{belgacem2025exploring,belgacem2026logistic}: global
existence, $C^{\infty}$-smoothness, uniform boundedness, and an
explicit Lipschitz bound of the same form as
\cite[Thm.~1]{belgacem2025exploring}.

\begin{theorem}[Global existence, smoothness, and boundedness]
\label{thm:well_posedness}
Consider the bias-corrected weighted-logistic ODE system
\begin{equation}
\dot{x}_i \;=\; F_i(\bx)
\;:=\; \kappa_i\,\fwgtbc(\bx) - \gamma_i\,x_i,
\qquad i = 1,\ldots,n,
\label{eq:F_def}
\end{equation}
with $\fwgtbc$ defined by \eqref{eq:bcw_model} and parameters
$\kappa_i,\gamma_i,\lambda > 0$, $\theta_{ij}>0$, $w_{ij}\in\R$.
Then:
\begin{enumerate}\itemsep1pt
\item[(a)] (\emph{Smoothness}) The right-hand side
$\mathbf{F}:\R^n\to\R^n$ is globally $C^{\infty}$.
\item[(b)] (\emph{Uniform Lipschitz bound}) The Jacobian
$\partial F_i/\partial x_l = \kappa_i\,\partial\fwgtbc/\partial x_l -
\gamma_i\delta_{il}$ is uniformly bounded:
\begin{equation}
\bigl|\partial F_i/\partial x_l(\bx)\bigr|
\;\le\; \tfrac{\kappa_i\lambda}{4}\,|w_{il}|
        \;+\; \gamma_i\,\delta_{il},
\qquad \forall\,\bx\in\R^n,
\label{eq:jacobian_bound}
\end{equation}
and the row-wise $\ell^{\infty}\!\to\!\ell^{\infty}$ Lipschitz constant
of $\mathbf{F}$ satisfies, under the convention $|w_{ij}|=1$ of
\Cref{rem:magnitude_rescaling},
\begin{equation}
L_F \;\le\; M \;:=\; \max_{i=1,\ldots,n}\!\Bigl(\kappa_i\,
\tfrac{\lambda\,m_i}{4} + \gamma_i\Bigr).
\label{eq:lipschitz_M}
\end{equation}
\item[(c)] (\emph{Global unique solution}) For every initial condition
$\bx(0)\in\R^n$, the system admits a unique solution $\bx(\cdot)\in
C^{\infty}([0,\infty);\R^n)$.
\item[(d)] (\emph{Positively invariant box}) The hyper-rectangle
$\Omega \;:=\; \prod_{i=1}^n \bigl[0,\kappa_i/\gamma_i\bigr]\subset
\R^n_{\ge 0}$ is positively invariant under \eqref{eq:F_def}; every
trajectory starting in $\Omega$ remains in $\Omega$ for all $t\ge 0$.
\end{enumerate}
\end{theorem}

\begin{proof}
\textbf{(a)} $\fwgtbc(\bx) = \sigma(\lambda S_i(\bx) + \lambda b_i)$
is a composition of the $C^{\infty}$ logistic $\sigma$ with an affine
function, hence $C^{\infty}$. So is each $F_i$. \textbf{(b)} A direct
computation using $\sigma'(z)\le 1/4$ gives
$|\partial\fwgtbc/\partial x_l| = |w_{il}|\lambda\,\sigma'(\cdot)\le
|w_{il}|\lambda/4$. Under $|w_{il}|=1$, the row-sum
$\sum_l |\partial F_i/\partial x_l|\le \kappa_i\lambda m_i/4 + \gamma_i
\le M$, yielding the $\ell^{\infty}\!\to\!\ell^{\infty}$ Lipschitz bound.
\textbf{(c)} Global Lipschitz right-hand side guarantees global
existence and uniqueness by the Picard--Lindelöf theorem on every
finite interval, extended to $[0,\infty)$ by continuation.
\textbf{(d)} For $x_i = 0$: $\dot{x}_i = \kappa_i\fwgtbc(\bx)\ge 0$,
so the lower face $\{x_i=0\}$ is non-out-going (and in fact strictly
inward by \Cref{cor:no_shutdown}). For $x_i = \kappa_i/\gamma_i$:
$\dot{x}_i = \kappa_i\fwgtbc(\bx)-\kappa_i\le \kappa_i\cdot 1-\kappa_i = 0$,
using $\fwgtbc\le 1$, so the upper face is also non-out-going.
\end{proof}

\Cref{thm:well_posedness} establishes that the bias-corrected
weighted-logistic framework inherits the global mathematical
guarantees of the product-of-logistics framework
\cite{belgacem2025exploring,belgacem2026logistic}: the explicit
Lipschitz constant \eqref{eq:lipschitz_M} matches the form
$L_F \le M = \max_i(\kappa_i\sum_j L_i^j + \gamma_i)$ of
\cite[Thm.~1]{belgacem2025exploring}, with each per-regulator
Lipschitz contribution $L_i^j = \lambda/4$ collapsed into the row-sum
$\lambda m_i/4$. The positively invariant box $\Omega$ matches the
biologically relevant concentration range $[0,\kappa_i/\gamma_i]$ of
each gene's saturation capacity.

\begin{corollary}[No expression shutdown]
\label{cor:no_shutdown}
At any state $\bx\in\R^n$ (in particular at $\bx = \mathbf{0}$),
the regulatory function $\fwgtbc$ is strictly positive:
\begin{equation}
\fwgtbc(\bx) \;>\; 0\qquad\forall\,\bx\in\R^n.
\label{eq:no_shutdown}
\end{equation}
At the zero state, $\fwgtbc(\mathbf{0}) =
1/\bigl(1+(2^{m_i}-1)\,e^{\lambda\sum_{j}w_{ij}\theta_{ij}}\bigr) > 0$;
consequently $\dot{x}_i\big|_{\bx=\mathbf{0}} = \kappa_i\,
\fwgtbc(\mathbf{0}) > 0$ and the zero-state is non-absorbing.
\end{corollary}

\begin{proof}
The denominator of \eqref{eq:bcw_compact} is positive and finite for
every $\bx\in\R^n$ since $e^{\lambda S_i(\bx)} > 0$, so $\fwgtbc(\bx)
>0$. At $\bx=\mathbf{0}$, $S_i(\mathbf{0}) = -\sum_j w_{ij}\theta_{ij}$
gives $\fwgtbc(\mathbf{0}) = 1/(1+(2^{m_i}-1)e^{\lambda\sum_j
w_{ij}\theta_{ij}})$.
\end{proof}

\Cref{cor:no_shutdown} parallels the basal-expression property of the
product-of-logistics framework~\cite{belgacem2026logistic}: like the
product (and unlike Hill-based models), the bias-corrected weighted
formulation has \emph{strictly positive basal output} and is therefore
free of the off-state trap that pathologically characterises Hill
models with non-integer cooperativity coefficients. This property is
essential for the bistable-escape, expression-recovery, and
basal-transcription analyses developed in~\cite{belgacem2026logistic}.

\section{Comparison with the Samuilik Weighted-Sum Formulation}
\label{sec:samuilik_comparison}

The bias-corrected weighted-logistic formulation \eqref{eq:bcw_model}
differs structurally from the Samuilik weighted-sum
model~\eqref{eq:samuilik_def} of \Cref{subsec:samuilik_def}. The
structural shortcomings of \eqref{eq:samuilik_def} relative to
product-of-logistics models were previously catalogued
in~\cite[\S 8.3]{belgacem2025exploring} along two complementary axes:
a critical-point pathology for repressors
(\cite[\S 8.3.1]{belgacem2025exploring}) and a threshold-scaling
pathology (\cite[\S 8.3.2]{belgacem2025exploring}). We extend that
analysis in two directions. First, we show that the bias-corrected
weighted-logistic formulation \emph{retains} the
single-sigmoid-per-gene structure of \eqref{eq:samuilik_def} (the very
feature that motivates weighted-sum models in the first place) while
\emph{resolving} both pathologies through regulator-specific
thresholds and the combinatorial bias $b_i$. Second, we make the
threshold-aggregation argument quantitative, showing that the standard
prescription $\theta_i^{\mathrm{S}} = \tfrac{1}{2}\sum_j
w_{ij}^{\mathrm{S}}$ collapses to a pure inflection condition with no
characteristic biological scale for canonical mixed regulation
(\Cref{subsec:mixed}).

\subsection{Threshold structure}
\label{subsec:threshold}

The two formulations differ in the placement and interpretation of
thresholds:
\begin{itemize}\itemsep2pt
\item In \eqref{eq:bcw_model}, each interaction $(i,j)$ has its own
threshold $\theta_{ij} > 0$, directly identifiable as the regulator
concentration at which interaction $(i,j)$ exerts half-maximal effect
on gene $i$. This quantity is independently determined from
single-regulator dose-response measurements (binding affinity $K_d$,
half-maximal effective concentration $\mathrm{EC}_{50}$, half-maximal
inhibitory concentration $\mathrm{IC}_{50}$).

\item In \eqref{eq:samuilik_def}, gene $i$ has a single shared threshold
$\theta_i^{\mathrm{S}}$ pooled across all regulators, computed as a
weighted aggregate. The biological meaning of this aggregate is opaque
when multiple regulators with disparate affinities are involved, and
the numerical value depends on the choice of weights, which are
themselves not directly measurable.
\end{itemize}

The structural distinction propagates to identifiability: thresholds
$\theta_{ij}$ in the bias-corrected weighted formulation can be
estimated from single-regulator experiments, while the Samuilik
threshold $\theta_i^{\mathrm{S}}$ requires simultaneous fitting of all
weights and the shared threshold from a single multi-regulator
dose-response surface, an inverse problem that is typically
ill-conditioned and admits weight-threshold tradeoffs that limit
parameter identifiability.

\subsection{Repression pathology of the Samuilik model}
\label{subsec:repression_pathology}

In the Samuilik model \eqref{eq:samuilik_def}, repression is encoded by a negative weight
$w<0$ inside an \emph{increasing} sigmoid:
\begin{equation}
f^{\mathrm{S},-}(x) \;=\; \frac{1}{1+e^{-\mu(wx-\theta)}}, \qquad w<0.
\label{eq:samuilik_repression}
\end{equation}
The midpoint of the sigmoid (where the output equals $1/2$ and the
slope is steepest) satisfies $wx_c - \theta = 0$, so
\[
x_c \;=\; \frac{\theta}{w} \;<\; 0
\qquad\text{when } \theta>0,\,w<0.
\]
The transition therefore occurs at a \emph{negative} concentration,
outside the biologically admissible domain $x\ge 0$. Concretely, with
the canonical choice $w=-1$, $\mu=4$, $\theta=3$, the function reduces
to $f^{\mathrm{S},-}(x) = 1/(1+e^{4(x+3)})$, satisfying
$f^{\mathrm{S},-}(x)\le 1/(1+e^{12})\approx 6\times 10^{-6}$ for
all $x\ge 0$. The function is essentially constant at zero throughout
the biologically relevant domain and never undergoes its sigmoidal
transition.

By contrast, the decreasing logistic $f^- (x,\theta,\lambda) =
1/(1+e^{-\lambda(\theta-x)})$ employed by the product
\eqref{eq:f_prod_def} and embedded in the bias-corrected weighted
form \eqref{eq:bcw_model} via the additive split
\eqref{eq:S_split} has its midpoint at $x = \theta>0$, the
biologically interpretable inhibition midpoint $\mathrm{IC}_{50}$,
and closely approximates the decreasing Hill function
$h^-(x,\theta,n) = \theta^n/(\theta^n+x^n)$ for matched parameters
$\lambda = n/\theta$ (see~\cite{belgacem2025exploring} for the
parameter-matching analysis).

\subsection{Mixed activation and repression}
\label{subsec:mixed}

Consider gene $i$ regulated by one activator $x_1$ ($w_{i1}=+1$) and
one repressor $x_2$ ($w_{i2}=-1$). In the bias-corrected weighted
framework \eqref{eq:bcw_model}, the regulatory function is
\begin{align}
\fwgtbc(x_1,x_2)
   &\;=\; \frac{1}{1+\exp\!\bigl(-\lambda\bigl[(x_1-\theta_{i1}) -
                                  (x_2-\theta_{i2}) -
                                  \lambda^{-1}\log 3\bigr]\bigr)}
\notag\\[3pt]
   &\;=\; \frac{1}{1+3\,\exp\!\bigl(-\lambda(x_1-x_2 +\theta_{i2}-
                                                  \theta_{i1})\bigr)}.
\label{eq:mixed_ours}
\end{align}
The thresholds $\theta_{i1},\theta_{i2}$ enter independently and are
determined from single-regulator dose-response curves for the
activator and the repressor.

In the Samuilik formulation, the same circuit reads
\begin{equation}
\fS(x_1,x_2) \;=\;
\frac{1}{1+\exp\!\bigl(-\mu_i(w_{i1}^{\mathrm{S}}x_1 +
                              w_{i2}^{\mathrm{S}}x_2 -
                              \theta_i^{\mathrm{S}})\bigr)},
\qquad w_{i1}^{\mathrm{S}}>0,\;w_{i2}^{\mathrm{S}}<0.
\label{eq:mixed_samuilik}
\end{equation}
With the standard choice $w_{i1}^{\mathrm{S}}=+1$,
$w_{i2}^{\mathrm{S}}=-1$ and the threshold prescription
$\theta_i^{\mathrm{S}} = \tfrac{1}{2}\sum_j w_{ij}^{\mathrm{S}} = 0$,
the function reduces to
$\fS(x_1,x_2) = 1/(1+e^{-\mu_i(x_1-x_2)})$. The half-maximal output
$\fS = 1/2$ is attained on the line $x_1=x_2$ regardless of any
characteristic concentration; biologically, this asserts that the
regulatory midpoint occurs at \emph{any} concentration at which the
activator and repressor are equal, in contradiction with the standard
biophysical picture in which the midpoint is anchored to specific
binding affinities $K_d^{\mathrm{act}}$ and $K_d^{\mathrm{rep}}$.

By contrast, \eqref{eq:mixed_ours} attains $\fwgtbc=1/4$ at the
critical point $x_1=\theta_{i1}, x_2=\theta_{i2}$, a configuration
anchored to two measurable biological scales.

\Cref{tab:comparison} consolidates the structural comparison.

\begin{table}[t]
\centering
\renewcommand{\arraystretch}{1.3}
\begin{tabularx}{\linewidth}{@{}lXXX@{}}
\toprule
\textbf{Property} & \textbf{Product-of-logistics} & \textbf{Bias-corrected weighted} & \textbf{Samuilik weighted-sum} \\
\midrule
Equation                & \eqref{eq:f_prod_def}                              & \eqref{eq:bcw_model}                                & \eqref{eq:samuilik_def}                                \\
Sigmoids per gene       & $m_i$ (one per regulator)                          & $1$                                                  & $1$                                                     \\
Threshold structure     & Regulator-specific $\theta_{ij}>0$                 & Regulator-specific $\theta_{ij}>0$                  & Single shared $\theta_i^{\mathrm{S}}$                  \\
Threshold meaning       & $\mathrm{EC}_{50}/\mathrm{IC}_{50}$ of regulator $j$ on $i$ & Idem                                       & Weighted aggregate, opaque                              \\
Direction encoding      & Multiplicative (form: $f^+,f^-$)                   & Additive (sign of $w_{ij}\in\R$ in exponent)        & Additive ($w_{ij}^{\mathrm{S}}\in\R$ in exponent)      \\
Repression midpoint     & $x_c = \theta_{ij}>0$                              & $x_c = \theta_{ij}>0$ (via additive split)          & $x_c = \theta/w<0$ (negative, pathological)            \\
Critical-point value    & $(1/2)^{m_i}$                                      & $(1/2)^{m_i}$                                       & $1/2$                                                  \\
Critical-point slope    & $\lambda/2^{m_i+1}$                                & $\lambda(2^{m_i}-1)/4^{m_i}$                        & $\mu_i/4$                                              \\
Slope ratio to product  & $1$                                                & $2-2^{1-m_i}$                                       & $2^{m_i+1}\mu_i/(4\lambda)$                            \\
Curvature $\partial^2/\partial x_l^2|_{\rm cp}$ & $0$                         & $\lambda^2(2^{m_i}-1)(2^{m_i}-2)/8^{m_i}$           & $0$                                                    \\
Identifiability         & Modular (single-regulator)                         & Modular (single-regulator)                          & Joint (multi-regulator inverse)                        \\
Bifurcation analysis    & Closed-form (symmetry $\sigma''(\theta)=0$)        & Quadratic-active for $m_i\ge 2$                     & Closed-form for $m_i=1$ only                           \\
\bottomrule
\end{tabularx}
\caption{Structural comparison of the three logistic-based GRN
formulations. The bias-corrected weighted formulation
\eqref{eq:bcw_model} retains the threshold structure and biological
identifiability of the product-of-logistics while compressing the
$m_i$ sigmoids into a single one. The price for this compression is a
loss of the structural symmetry $\sigma''(\theta)=0$ that drives
analytical bifurcation theorems for the product (last row).}
\label{tab:comparison}
\end{table}

\section{Functional Equivalence Analysis}
\label{sec:equivalence}

This section contains the principal mathematical results of the paper.
We show that the bias-corrected weighted formulation \eqref{eq:bcw_model}
matches the product-of-logistics \eqref{eq:f_prod_def} at three
biologically meaningful reference configurations
(\Cref{thm:three_point}); we derive an exact algebraic identity for
the discrepancy between the two formulations
(\Cref{thm:algebraic_identity}); we compute the closed-form ratio of
local sensitivities at the critical point
(\Cref{prop:slope_ratio}); we prove that the structural symmetry
$\sigma''(\theta)=0$ enjoyed by the product formulation is broken by
the bias correction for $m_i\ge 2$
(\Cref{prop:curvature_mismatch}); and we prove that no
state-independent constant shift can yield global equivalence
(\Cref{prop:no_global_eq}).

\subsection{Three-point matching}
\label{subsec:three_point}

\begin{theorem}[Three-point matching]
\label{thm:three_point}
The bias-corrected weighted formulation $\fwgtbc$ in
\eqref{eq:bcw_model} matches the product-of-logistics $\fprod$ in
\eqref{eq:f_prod_def} at the following three configurations:

\textbf{(a) Critical point.}
At $\bx = \bx_i^{(\mathrm{cp})}$, i.e.\ $x_j = \theta_{ij}$ for every
$j\in\Acal_i\cup\Rcal_i$, both functions equal $(1/2)^{m_i}$.

\textbf{(b) Low-steepness limit.}
As $\lambda\to 0^+$, both functions tend uniformly on compact subsets
of $\R^n$ to $(1/2)^{m_i}$.

\textbf{(c) Saturation limits.}
As every individual sigmoid factor in \eqref{eq:f_prod_def} approaches
$1$ (full activation, no repression),
both $\fprod\to 1$ and $\fwgtbc\to 1$. Symmetrically, as every factor
approaches $0$, both functions tend to $0$.
\end{theorem}

\begin{proof}
\textbf{(a)} At the critical point $S_i = 0$, so by
\eqref{eq:bcw_compact},
$\fwgtbc = 1/(1+(2^{m_i}-1)) = 1/2^{m_i}$, which matches $\fprod$ by
\Cref{rem:critical_point}.

\textbf{(b)} As $\lambda\to 0^+$ at fixed $\bx$, the exponent in
\eqref{eq:bcw_model} satisfies $-\lambda(S_i+b_i) = -\lambda S_i +
\log(2^{m_i}-1)\to\log(2^{m_i}-1)$. Therefore
\[
\fwgtbc(\bx) \;\to\; \frac{1}{1+e^{\log(2^{m_i}-1)}} \;=\;
\frac{1}{2^{m_i}} \;=\; \lim_{\lambda\to 0^+}\fprod(\bx),
\]
where the second equality uses \Cref{lem:limits}(a). Convergence is
uniform on compacts: fix a compact $K\subset\R^n$ and set
$M_K := \sup_{\bx\in K}\sum_{j=1}^n|w_{ij}|\,|x_j-\theta_{ij}|<\infty$.
Then $|\lambda S_i(\bx)|\le \lambda M_K$ for all $\bx\in K$. Both
$\fwgtbc$ and $\fprod$ depend on $\bx$ only through $\lambda S_i$ (and
the individual $\lambda(x_j-\theta_{ij})$ for the product), and the
sigmoidal Lipschitz bound $|\sigma'|\le 1/4$ gives
$|\fwgtbc(\bx) - 2^{-m_i}|\le \tfrac{1}{4}\lambda M_K$ and an analogous
bound for $\fprod$, both tending to zero uniformly on $K$ as
$\lambda\to 0^+$.

\textbf{(c)} Under the joint positive saturation condition $u_j\to
+\infty$ for every $j\in[m_i]$, $\lambda S_i = \sum_j u_j\to +\infty$,
hence $(2^{m_i}-1)e^{-\lambda S_i}\to 0$ and $\fwgtbc\to 1$. The
limit $\fprod\to 1$ follows from \Cref{lem:limits}(b). The case of
joint negative saturation $u_j\to -\infty$ is symmetric and gives
$\fprod,\fwgtbc\to 0$ by \Cref{lem:limits}(c).
\end{proof}

\Cref{thm:three_point} establishes that the bias-corrected weighted
formulation \emph{registers} with the product-of-logistics at the
three reference configurations: the half-input critical point that
characterises the regulatory threshold, the low-steepness limit that
characterises non-cooperative response, and the saturation limits
that characterise full activation/repression. The bias term
\eqref{eq:bcw_model} is the unique state-independent constant that
achieves the critical-point match (\Cref{prop:bias_canonicity});
since the low-steepness limit reduces algebraically to the same
condition $\sigma(\lambda b_i) = 1/2^{m_i}$ and the saturation
limits hold automatically for any constant bias, the critical-point
condition therefore implies the full three-point match.

\subsection{Exact algebraic discrepancy formula}
\label{subsec:discrepancy_formula}

We now express the discrepancy $\fprod - \fwgtbc$ in closed form. The
key technical device is a change of variables that absorbs the weight
sign:
\begin{equation}
u_j \;:=\; w_{ij}\,\lambda(x_j - \theta_{ij})
\;=\;
\begin{cases}
\;\;\lambda(x_j-\theta_{ij}), & j\in\Acal_i,\\
-\lambda(x_j-\theta_{ij}) = \lambda(\theta_{ij}-x_j),
                              & j\in\Rcal_i,\\
\;\;0,                         & w_{ij}=0.
\end{cases}
\label{eq:u_def}
\end{equation}
Under this convention, both activators and repressors become
sigmoid factors of the form $\sigma(u_j) = 1/(1+e^{-u_j})$ in the
product formulation:
\begin{equation}
\fprod \;=\; \prod_{j:\, w_{ij}\neq 0} \sigma(u_j)
\;=\; \prod_{j:\, w_{ij}\neq 0} \frac{e^{u_j}}{1+e^{u_j}}
\;=\; \frac{e^{U_i}}{\prod_{j:\,w_{ij}\neq 0}(1+e^{u_j})},
\label{eq:f_prod_in_u}
\end{equation}
where $U_i := \sum_{j:\,w_{ij}\neq 0} u_j = \lambda S_i$. Similarly,
the bias-corrected weighted form rewrites as
\begin{equation}
\fwgtbc \;=\; \frac{e^{U_i}}{(2^{m_i}-1) + e^{U_i}}.
\label{eq:f_bc_in_u}
\end{equation}

Let $[m_i] = \{1,\ldots,m_i\}$ denote the index set of the regulators
of gene $i$ (re-indexed for notational simplicity). For a subset
$T\subseteq[m_i]$ write $u_T := \sum_{j\in T}u_j$ and let
$\mathcal{T}_i := \{T\subseteq [m_i] : T\neq\emptyset,\,T\neq [m_i]\}$
denote the collection of proper non-empty subsets of $[m_i]$. The
cardinality of $\mathcal{T}_i$ is $2^{m_i}-2$.

\begin{theorem}[Algebraic identity for the discrepancy]
\label{thm:algebraic_identity}
With $u_j$ as in \eqref{eq:u_def} and $U_i = \sum_j u_j = \lambda S_i$,
\begin{equation}
\eqbox{
\fprod(\bx) - \fwgtbc(\bx)
\;=\;
-\;\frac{e^{U_i}\;\displaystyle\sum_{T\in\mathcal{T}_i}
        \bigl(e^{u_T}-1\bigr)}
       {\displaystyle
        \Bigl[\prod_{j\in[m_i]}(1+e^{u_j})\Bigr]\,
        \bigl[(2^{m_i}-1) + e^{U_i}\bigr]}.}
\label{eq:algebraic_identity}
\end{equation}
\end{theorem}

\begin{proof}
By \eqref{eq:f_prod_in_u}--\eqref{eq:f_bc_in_u},
\[
\fprod-\fwgtbc \;=\;
e^{U_i}\!\left[
\frac{1}{\prod_j(1+e^{u_j})} \;-\; \frac{1}{(2^{m_i}-1)+e^{U_i}}
\right]
\;=\;
\frac{e^{U_i}\,\bigl[(2^{m_i}-1)+e^{U_i} - \prod_j(1+e^{u_j})\bigr]}
     {\prod_j(1+e^{u_j})\;\bigl[(2^{m_i}-1)+e^{U_i}\bigr]}.
\]
Expand the product using the elementary identity
\[
\prod_{j=1}^{m_i}(1+e^{u_j})
\;=\; \sum_{T\subseteq[m_i]} e^{u_T}
\;=\; 1 \;+\; \sum_{T\in\mathcal{T}_i} e^{u_T} \;+\; e^{U_i},
\]
where the last equality separates the empty subset $T=\emptyset$
(which contributes $e^0 = 1$) and the full subset $T=[m_i]$ (which
contributes $e^{u_{[m_i]}} = e^{U_i}$) from the $2^{m_i}-2$ proper
non-empty subsets. Substituting,
\[
(2^{m_i}-1)+e^{U_i} - \prod_j(1+e^{u_j})
\;=\; (2^{m_i}-1) + e^{U_i} - 1 - \!\!\sum_{T\in\mathcal{T}_i}e^{u_T} - e^{U_i}
\;=\; (2^{m_i}-2) - \sum_{T\in\mathcal{T}_i}e^{u_T}.
\]
Since $|\mathcal{T}_i| = 2^{m_i}-2$, this rewrites as
$-\sum_{T\in\mathcal{T}_i}\bigl(e^{u_T}-1\bigr)$, completing the
proof.
\end{proof}

The identity \eqref{eq:algebraic_identity} provides the
\emph{additive} form of the discrepancy. A multiplicative
counterpart is obtained directly from \eqref{eq:f_prod_in_u} and
\eqref{eq:f_bc_in_u}:

\begin{corollary}[Ratio form of the discrepancy]
\label{cor:ratio}
With $u_j$, $U_i = \lambda S_i$ as in \Cref{thm:algebraic_identity},
\begin{equation}
\frac{\fprod(\bx)}{\fwgtbc(\bx)}
\;=\;
\frac{(2^{m_i}-1) + e^{\lambda S_i}}
     {\displaystyle\prod_{j\in [m_i]}\bigl(1+e^{u_j}\bigr)}.
\label{eq:ratio_identity}
\end{equation}
This ratio equals $1$ at the critical point ($\bu=\mathbf{0}$) and as
$u_j\to+\infty$ for all $j$ (joint positive saturation), and tends to
$2^{m_i}-1$ as $u_j\to-\infty$ for all $j$ (joint negative saturation).
\end{corollary}

\begin{proof}
Dividing \eqref{eq:f_prod_in_u} by \eqref{eq:f_bc_in_u} gives
\eqref{eq:ratio_identity}. At $\bu = \mathbf{0}$: numerator
$(2^{m_i}-1)+1 = 2^{m_i}$, denominator $2^{m_i}$, ratio $1$. As
$u_j\to+\infty$ all, both numerator and denominator are $\sim e^{U_i}$,
ratio $\to 1$. As $u_j\to -\infty$ all, $e^{U_i}\to 0$ and
$\prod(1+e^{u_j})\to 1$, giving ratio $\to (2^{m_i}-1)/1 = 2^{m_i}-1$.
\end{proof}

\Cref{cor:ratio} is the natural form for \emph{relative-error} bounds:
the ratio characterises by what factor the product overshoots or
undershoots the bias-corrected weighted formulation, complementing
\eqref{eq:algebraic_identity}, which quantifies absolute pointwise
deviation. In the negative-saturation regime both functions decay to
zero exponentially at the same rate $\sim e^{m_i u}$, but the
product has a leading-order coefficient $(2^{m_i}-1)$ times larger
than the bias-corrected weighted form: $\fprod/\fwgtbc\to (2^{m_i}-1)$
captures the multiplicative gap between the two formulations'
off-state output. For $m_i = 2,3,4$ the asymptotic ratios are
$3,7,15$ respectively, a substantial relative discrepancy that should
be accounted for in controller designs that operate at low
expression levels.

The identity \eqref{eq:algebraic_identity} provides a complete
characterisation of the discrepancy. Three immediate consequences are
worth recording.

\begin{corollary}[Vanishing set]
\label{cor:vanishing}
$\fprod(\bx) = \fwgtbc(\bx)$ if and only if
\begin{equation}
\Phi_i(\bu) \;:=\; \sum_{T\in\mathcal{T}_i}\bigl(e^{u_T}-1\bigr)
\;=\; \sum_{T\in\mathcal{T}_i} e^{u_T} - (2^{m_i}-2) \;=\; 0.
\label{eq:Phi_zero}
\end{equation}
Three natural classes of configurations satisfy this:
\begin{enumerate}
\item[(i)] \emph{Critical point:} at $\bu = \mathbf{0}$, every individual
            term $e^{u_T}-1$ vanishes, so $\Phi_i = 0$ \emph{exactly};
            this is the unique point in $\R^{m_i}$ at which all
            $2^{m_i}-2$ summands vanish simultaneously.
\item[(ii)] \emph{Joint saturation, $u_j\to+\infty$ all:} $\Phi_i\to 0$
            \emph{asymptotically} (precise rates in
            \Cref{prop:asymptotic_rates} below).
\item[(iii)] \emph{Joint saturation, $u_j\to-\infty$ all:} likewise
            $\Phi_i\to 0$ asymptotically.
\end{enumerate}
For $m_i\ge 2$, the equation $\Phi_i(\bu) = 0$ defines an
$(m_i-1)$-dimensional real-analytic submanifold of $\R^{m_i}$ on which
the discrepancy vanishes exactly through cancellation between subsets
$T$ with $u_T > 0$ (positive contribution $e^{u_T}-1>0$) and subsets
with $u_T < 0$ (negative contribution). The critical point is the
unique point of this submanifold at which each summand individually
vanishes; elsewhere, the manifold avoids the orthants
$\{u_j\ge 0\;\forall j\}\setminus\{\mathbf{0}\}$ and
$\{u_j\le 0\;\forall j\}\setminus\{\mathbf{0}\}$ entirely, since on
those sets every summand is non-negative (resp.\ non-positive) and at
least one is strictly so.
\end{corollary}

\begin{proof}
The first claim follows from \eqref{eq:algebraic_identity}: the
denominator and the prefactor $e^{U_i}$ are strictly positive, so the
discrepancy vanishes iff the inner sum vanishes. Points (i)--(iii) are
immediate. For the manifold claim, $\Phi_i$ is real-analytic on
$\R^{m_i}$, and a direct computation gives, at $\bu = \mathbf{0}$,
\[
\partial_{u_j}\Phi_i\bigl|_{\mathbf{0}}
\;=\; \sum_{T\in\mathcal{T}_i:\, j\in T} e^{u_T}\bigl|_{\mathbf{0}}
\;=\; \bigl|\{T\in\mathcal{T}_i : j\in T\}\bigr|
\;=\; 2^{m_i-1}-1 \;>\; 0,
\]
since the proper non-empty subsets of $[m_i]$ containing $j$ are in
bijection with subsets of $[m_i]\setminus\{j\}$ other than the full
set, of which there are $2^{m_i-1}-1$. The gradient $\nabla\Phi_i$ is
therefore nonzero at $\mathbf{0}$, and by the implicit function theorem
$\{\Phi_i = 0\}$ is an embedded $(m_i-1)$-dimensional analytic
submanifold of $\R^{m_i}$ in a neighbourhood of $\mathbf{0}$, extending
globally on the full level set. The exclusion of the strictly-positive
and strictly-negative orthants is immediate from $e^{u_T}-1$ having
the sign of $u_T$.
\end{proof}

\begin{corollary}[Symmetric-input case]
\label{cor:symmetric}
If all $u_j = u$ for a common scalar $u\in\R$,
\eqref{eq:algebraic_identity} reduces to
\begin{equation}
\fprod-\fwgtbc \;=\;
-\;\frac{e^{m_i u}\,\sum_{k=1}^{m_i-1}\binom{m_i}{k}\bigl(e^{ku}-1\bigr)}
       {(1+e^u)^{m_i}\,\bigl[(2^{m_i}-1) + e^{m_i u}\bigr]}.
\label{eq:symmetric_identity}
\end{equation}
\end{corollary}

\Cref{fig:comparison} illustrates this symmetric-case discrepancy.
The maximum absolute discrepancy over $u\in\R$ grows with $m_i$:
numerical optimisation gives $\max_u|\fprod-\fwgtbc| \approx 0.20$
for $m_i=2$, $\approx 0.39$ for $m_i=3$, and $\approx 0.54$ for
$m_i=4$, attained at $u\approx 1.4, 1.3, 1.25$ respectively. The
discrepancy decays as $u\to\pm\infty$ but does not vanish identically
between the reference points.

\begin{figure}[t]
\centering
\includegraphics[width=0.95\linewidth]{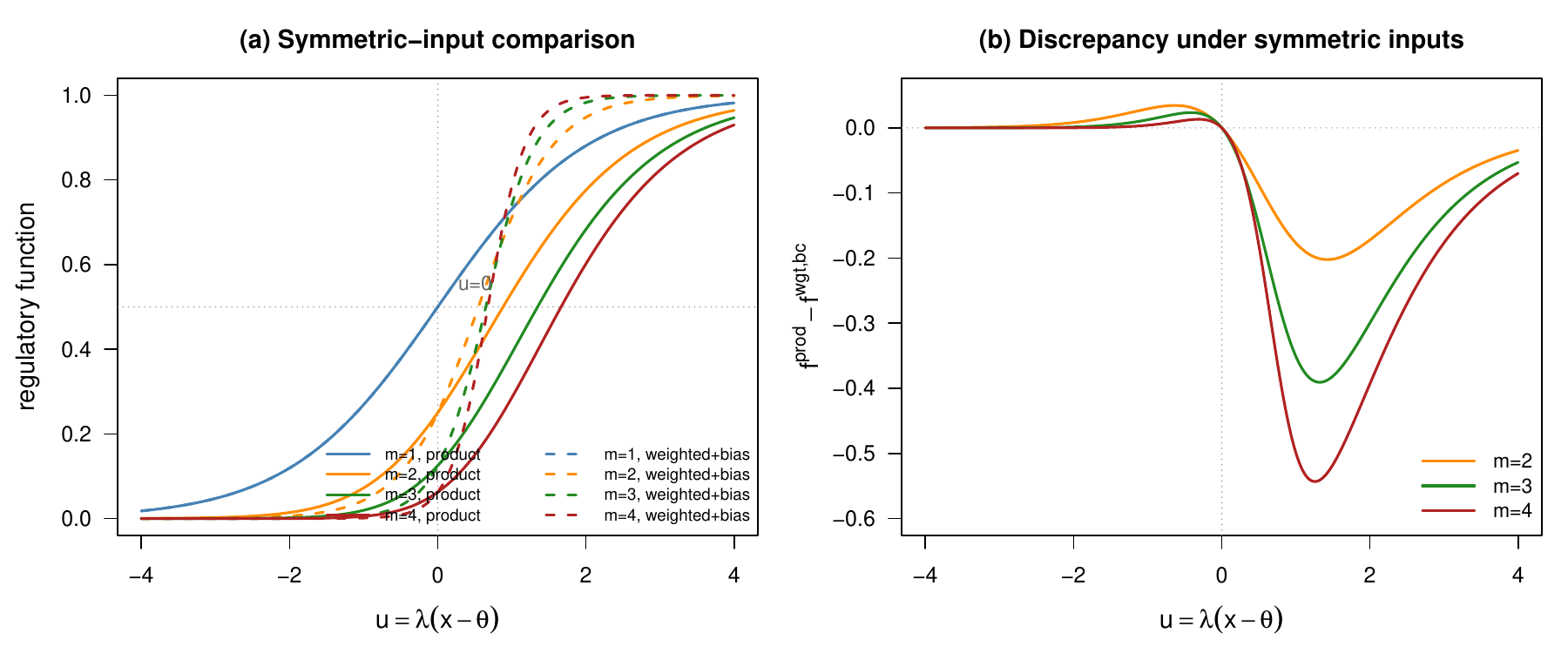}
\caption{Symmetric-input comparison ($u_j = u$, all activators) of the
product-of-logistics $\fprod = \sigma(u)^{m_i}$ and the bias-corrected
weighted-logistic $\fwgtbc = \sigma(m_i u - \log(2^{m_i}-1))$.
\textbf{(a)} Output curves for $m_i = 1,2,3,4$. Both formulations pass
through the critical-point value $1/2^{m_i}$ at $u=0$ and saturate
to $0$ and $1$ in the limits $u\to\mp\infty$. The bias-corrected
weighted form (dashed) is uniformly steeper than the product (solid)
for $m_i\ge 2$; the slope ratio at $u=0$ is $2-2^{1-m_i}$
(\Cref{prop:slope_ratio}). \textbf{(b)} Pointwise discrepancy
$\fprod - \fwgtbc$ from \eqref{eq:symmetric_identity}. The discrepancy
vanishes at $u=0$ and as $u\to\pm\infty$, and grows in absolute value
with $m_i$.}
\label{fig:comparison}
\end{figure}

\begin{corollary}[Pointwise sign of the discrepancy in the symmetric case]
\label{cor:sign}
In the symmetric configuration, $\fprod(\bx)-\fwgtbc(\bx)$ has the
opposite sign to $u$: the product overshoots the bias-corrected weighted
form for $u<0$ and undershoots it for $u>0$.
\end{corollary}

\begin{proof}
The denominator of \eqref{eq:symmetric_identity} is positive. The
numerator's sign is opposite to that of $\sum_{k=1}^{m_i-1}
\binom{m_i}{k}(e^{ku}-1)$, which is positive for $u>0$
(every term positive), zero at $u=0$, and negative for $u<0$.
\end{proof}

\begin{proposition}[Sharp asymptotic decay rates of the discrepancy]
\label{prop:asymptotic_rates}
In the symmetric configuration ($u_j = u$ for all $j$), the
discrepancy admits the leading-order asymptotic estimates
\begin{equation}
\fprod(\bx) - \fwgtbc(\bx)
\;\sim\;
\begin{cases}
\,-\,m_i\,e^{-u}, & u\to +\infty,\\[3pt]
\,\dfrac{2^{m_i}-2}{2^{m_i}-1}\,e^{m_i u},
                  & u\to -\infty.
\end{cases}
\label{eq:asymptotic_rates}
\end{equation}
In the general (non-symmetric) configuration with $u_j\to+\infty$ for
all $j$, the discrepancy is dominated by the contribution of the
$(m_i-1)$-element subsets:
\begin{equation}
\fprod(\bx)-\fwgtbc(\bx) \;=\; -\,\Bigl(\sum_{l\in[m_i]}e^{-u_l}\Bigr)
                              \;+\; O\!\Bigl(e^{-\min_{l\neq l'}(u_l+u_{l'})}\Bigr).
\label{eq:asymptotic_general}
\end{equation}
\end{proposition}

\begin{proof}
\textbf{Symmetric case, $u\to+\infty$.} Substitute $u_j = u$ into
\eqref{eq:symmetric_identity}. The numerator's inner sum is dominated
by the largest exponent, $k=m_i-1$: $\binom{m_i}{m_i-1}(e^{(m_i-1)u}
-1)\sim m_i\,e^{(m_i-1)u}$. The full numerator is therefore
$\sim m_i e^{m_i u}\cdot e^{(m_i-1)u} = m_i e^{(2m_i-1)u}$. The
denominator: $(1+e^u)^{m_i}\sim e^{m_i u}$ and
$(2^{m_i}-1)+e^{m_i u}\sim e^{m_i u}$, so the denominator is
$\sim e^{2m_i u}$. The ratio gives $-m_i e^{-u}$, as claimed.

\textbf{Symmetric case, $u\to-\infty$.} The numerator's inner sum is
$\sum_{k=1}^{m_i-1}\binom{m_i}{k}(e^{ku}-1)\to -(2^{m_i}-2)$ since
$e^{ku}\to 0$ for $k\ge 1$ and $\sum_k\binom{m_i}{k} = 2^{m_i}-2$ for
$1\le k\le m_i-1$. The numerator is therefore
$\sim e^{m_i u}\cdot(-(2^{m_i}-2))$, and the denominator
$(1+e^u)^{m_i}\to 1$ while $(2^{m_i}-1)+e^{m_i u}\to 2^{m_i}-1$. The
ratio gives $\frac{2^{m_i}-2}{2^{m_i}-1}e^{m_i u}$ (the minus sign in
\eqref{eq:symmetric_identity} cancels the minus from the inner-sum
limit).

\textbf{General case, $u_j\to+\infty$.} In
\eqref{eq:algebraic_identity}, the dominant terms in the numerator's
inner sum are the $(m_i-1)$-element subsets $T = [m_i]\setminus\{l\}$,
contributing $e^{U_i - u_l}$ for each $l\in[m_i]$. Lower-cardinality
subsets contribute strictly smaller exponents. The denominator
$\prod_j(1+e^{u_j})\sim e^{U_i}$ and $(2^{m_i}-1)+e^{U_i}\sim e^{U_i}$.
Combining gives \eqref{eq:asymptotic_general}.
\end{proof}

\Cref{prop:asymptotic_rates} sharpens \Cref{cor:vanishing}: the
discrepancy vanishes exponentially fast in the saturation limits, with
exponential rate $1$ (independent of $m_i$) in the $u\to+\infty$
direction and rate $m_i$ in the $u\to-\infty$ direction. The latter
rate is faster precisely because the bias correction makes
\emph{both} formulations decay at the natural rate
$\fprod, \fwgtbc \sim e^{m_i u}$; in the $u\to+\infty$ direction the
product approaches $1$ at the slower individual-factor rate
$1-m_i e^{-u}+O(e^{-2u})$, while the weighted form approaches $1$ at
the faster rate $1-(2^{m_i}-1)e^{-m_i u}+O(e^{-2m_i u})$. The
discrepancy is therefore dominated by the slower decay of the product
in the positive-saturation regime.

\subsection{Local sensitivity at the critical point}
\label{subsec:slope_ratio}

We now compute the ratio of partial derivatives of the two regulatory
functions at the critical point.

\begin{proposition}[Slope-ratio formula]
\label{prop:slope_ratio}
At the critical point $\bx_i^{(\mathrm{cp})}$ of gene $i$,
\begin{align}
\partial_{x_l}\fprod\big|_{\mathrm{cp}}
   \;&=\; \frac{\sgn(w_{il})\,\lambda}{2^{m_i+1}},
\label{eq:slope_prod}\\[2pt]
\partial_{x_l}\fwgtbc\big|_{\mathrm{cp}}
   \;&=\; \frac{w_{il}\,\lambda\,(2^{m_i}-1)}{4^{m_i}},
\label{eq:slope_wgt}
\end{align}
for any regulator $l$ of gene $i$. Consequently,
\begin{equation}
\eqbox{
\frac{|\partial_{x_l}\fwgtbc|_{\mathrm{cp}}|}
     {|\partial_{x_l}\fprod|_{\mathrm{cp}}|}
\;=\;
\frac{2(2^{m_i}-1)}{2^{m_i}}
\;=\; 2 - 2^{1-m_i}.}
\label{eq:slope_ratio}
\end{equation}
The ratio equals $1$ for $m_i = 1$, increases monotonically in $m_i$,
and approaches $2$ as $m_i\to\infty$.
\end{proposition}

\begin{proof}
\textbf{Slope of $\fprod$.} By \eqref{eq:f_prod_def}, with $u_j$ as
in \eqref{eq:u_def},
\[
\partial_{x_l}\fprod \;=\; w_{il}\,\lambda\,\sigma'(u_l)
                          \!\!\prod_{\substack{j\in[m_i]\\j\neq l}}
                          \!\!\sigma(u_j),
\]
since each factor of $\fprod$ depends on $x_l$ only through $u_l$.
At the critical point all $u_j=0$, so $\sigma(0)=\half$ and
$\sigma'(0)=\sigma(0)(1-\sigma(0)) = 1/4$, giving
$\partial_{x_l}\fprod = w_{il}\,\lambda\cdot 1/4\cdot (1/2)^{m_i-1}
= w_{il}\,\lambda/2^{m_i+1} = \sgn(w_{il})\,\lambda/2^{m_i+1}$.

\textbf{Slope of $\fwgtbc$.} Write $\fwgtbc = \sigma(\lambda S_i +
\log(2^{m_i}-1)\cdot(-1)) = \sigma(\lambda S_i - \log(2^{m_i}-1))$.
Then $\partial_{x_l}\fwgtbc = w_{il}\,\lambda\,
\sigma'(\lambda S_i - \log(2^{m_i}-1))$.
At the critical point $S_i = 0$ and $\sigma'(-\log(2^{m_i}-1)) =
\sigma(-\log(2^{m_i}-1))(1-\sigma(-\log(2^{m_i}-1)))$, with
\[
\sigma(-\log(2^{m_i}-1)) = \frac{1}{1+e^{\log(2^{m_i}-1)}}
                         = \frac{1}{2^{m_i}}.
\]
Hence $\sigma'(-\log(2^{m_i}-1)) = (1/2^{m_i})(1-1/2^{m_i}) =
(2^{m_i}-1)/4^{m_i}$, giving \eqref{eq:slope_wgt}.

\textbf{Ratio.} Taking absolute values,
$|\partial_{x_l}\fwgtbc|/|\partial_{x_l}\fprod| =
((2^{m_i}-1)/4^{m_i})/(1/2^{m_i+1})
= 2^{m_i+1}(2^{m_i}-1)/4^{m_i}
= 2(2^{m_i}-1)/2^{m_i} = 2-2^{1-m_i}$.

The ratio equals $2-1=1$ for $m_i=1$ and approaches $2$ as $m_i\to
\infty$. Monotonicity: $\frac{d}{dm}(2-2^{1-m}) = 2^{1-m}\log 2 > 0$.
\end{proof}

\Cref{prop:slope_ratio} quantifies the structural difference between
the two formulations at the critical point: the bias-corrected weighted
form is locally up to twice as sensitive as the product, with the
sharpening factor exactly $2-2^{1-m_i}$. For typical biological
networks with $m_i\in\{2,3,4\}$, the ratios are $3/2$, $7/4$,
$15/8$. This monotone sharpening is biologically interpretable: the
product formulation enforces consensus among regulators
(multiplicative occupancy), while the bias-corrected weighted form
sums log-odds (additive evidence), producing a sharper aggregate
response.

\begin{remark}[Comparison of slopes across all three formulations]
\label{rem:slope_three_way}
At the critical point of gene $i$ (when meaningfully defined for each
formulation), the magnitudes of the partial derivatives with respect
to any regulator $l$ are:
\begin{equation}
\bigl|\partial_{x_l}\fprod\bigr|_{\rm cp} = \frac{\lambda}{2^{m_i+1}},
\quad
\bigl|\partial_{x_l}\fwgtbc\bigr|_{\rm cp} = \frac{\lambda(2^{m_i}-1)}{4^{m_i}},
\quad
\bigl|\partial_{x_l}\fS\bigr|_{\rm cp} = \frac{\mu_i}{4}.
\label{eq:three_slopes}
\end{equation}
With the canonical parameter matching $\mu_i = \lambda$, the
Samuilik-to-product slope ratio is exactly $2^{m_i-1}$, growing
\emph{exponentially} in $m_i$, whereas the bias-corrected weighted
ratio $2-2^{1-m_i}$ stays bounded above by $2$. The Samuilik-to-bcw
ratio is therefore
\begin{equation}
\frac{\bigl|\partial_{x_l}\fS\bigr|_{\rm cp}}
     {\bigl|\partial_{x_l}\fwgtbc\bigr|_{\rm cp}}
\;=\; \frac{2^{m_i-1}}{2-2^{1-m_i}}
\;=\; \frac{2^{m_i-2}}{1-2^{-m_i}}
\;\sim\; 2^{m_i-2}\qquad(m_i\to\infty),
\label{eq:samuilik_bcw_ratio}
\end{equation}
taking the values $4/3,\,16/7,\,64/15$ for $m_i=2,3,4$ respectively.
This widening disparity quantifies the divergence of the Samuilik
formulation from the product as $m_i$ increases.
\end{remark}

\subsection{Curvature mismatch and broken structural symmetry}
\label{subsec:curvature}

A more delicate structural difference between the two formulations
emerges at second order.

\begin{proposition}[Curvature mismatch at the critical point]
\label{prop:curvature_mismatch}
At the critical point of gene $i$,
\begin{align}
\partial_{x_l}^2 \fprod\big|_{\mathrm{cp}} \;&=\; 0 \qquad
\text{for all } m_i\ge 1,\label{eq:fprod_zero_curv}\\[2pt]
\partial_{x_l}^2 \fwgtbc\big|_{\mathrm{cp}} \;&=\;
   \frac{w_{il}^2\,\lambda^2\,(2^{m_i}-1)(2^{m_i}-2)}{8^{m_i}}.
\label{eq:fwgt_curv}
\end{align}
The second derivative of $\fwgtbc$ vanishes iff $m_i=1$, while that
of $\fprod$ vanishes for every $m_i\ge 1$.
\end{proposition}

\begin{proof}
For $\fprod$, holding $x_j = \theta_{ij}$ fixed for $j\neq l$ leaves
\[
\fprod(x_l)\big|_{\rm{others}=\theta} \;=\; \sigma(u_l)\cdot \half^{m_i-1},
\]
so $\partial_{x_l}^2\fprod = w_{il}^2\lambda^2\sigma''(u_l)\cdot(\half)^{m_i-1}$.
At $u_l=0$, $\sigma''(0) = \sigma'(0)(1-2\sigma(0)) = (1/4)(1-1) = 0$,
proving \eqref{eq:fprod_zero_curv}.

For $\fwgtbc$, holding $x_j=\theta_{ij}$ fixed for $j\neq l$ leaves
$\fwgtbc(x_l) = \sigma(u_l - \log(2^{m_i}-1))$ (since under the
single-variable restriction $\lambda S_i = w_{il}\lambda(x_l-
\theta_{il}) = u_l$), and
$\partial_{x_l}^2\fwgtbc = w_{il}^2\lambda^2 \sigma''(-\log(2^{m_i}-1))$.
Now
\[
\sigma''(z) = \sigma'(z)(1-2\sigma(z))
            = \sigma(z)(1-\sigma(z))(1-2\sigma(z)),
\]
and at $z=-\log(2^{m_i}-1)$ with $\sigma(z) = 1/2^{m_i}$,
\[
\sigma''(z) = \frac{1}{2^{m_i}}\cdot\frac{2^{m_i}-1}{2^{m_i}}\cdot
              \left(1 - \frac{2}{2^{m_i}}\right)
            = \frac{(2^{m_i}-1)(2^{m_i}-2)}{8^{m_i}},
\]
yielding \eqref{eq:fwgt_curv}. This vanishes iff $m_i=1$, since
$2^{m_i}-2>0$ for $m_i\ge 2$.
\end{proof}

The structural asymmetry exposed by \Cref{prop:curvature_mismatch}
has direct consequences for normal-form bifurcation analysis. In the
product-of-logistics framework, the closed-form pitchfork bifurcation
of the symmetric balanced toggle switch~\cite{gardner2000construction}
and the supercritical Hopf
bifurcation of the symmetric balanced
repressilator~\cite{elowitz2000synthetic} both rely on the
identity $f^{-\prime\prime}(\theta) = 0$
(equivalently $\sigma''(0)=0$), which forces the leading-order
nonlinearity in the centre-manifold reduction to be cubic and
stabilising~\cite{belgacem2025exploring,belgacem2026logistic}. The
bias-corrected weighted formulation, by introducing the shift
$-\log(2^{m_i}-1)$ that places the sigmoid evaluation away from the
inflection point of $\sigma$, generically activates the quadratic
term and may render the corresponding bifurcations \emph{subcritical}
or \emph{transcritical}. This is a substantive limitation of the
weighted formulation for analytical bifurcation theory, and a strong
reason to retain the product formulation when normal-form analysis
on canonical motifs is the analytical objective.

\subsection{No state-independent global equivalence}
\label{subsec:no_global}

A natural question is whether a more elaborate state-independent shift
could eliminate the discrepancy of \Cref{thm:algebraic_identity}
globally. The answer is negative.

\begin{proposition}[Non-existence of state-independent global equivalence]
\label{prop:no_global_eq}
Under the unit-magnitude convention of \Cref{rem:magnitude_rescaling},
fix $m_i\ge 2$, $\lambda>0$, and a sign pattern
$w_{i\cdot}\in\{-1,+1\}^{m_i}$. There is no constant $b\in\R$ such that
\begin{equation}
\sigma\bigl(\lambda S_i(\bx) + \lambda b\bigr)
\;=\; \fprod(\bx)
\qquad\text{for all }\bx\in\R^n.
\label{eq:no_global}
\end{equation}
In particular, the bias \eqref{eq:bcw_model} is the unique constant
$b\in\R$ realising the critical-point match
$\fwgtbc(\bx_i^{(\mathrm{cp})}) = \fprod(\bx_i^{(\mathrm{cp})})$, and
no other constant achieves global equivalence either.
\end{proposition}

\begin{proof}
Suppose, for contradiction, that \eqref{eq:no_global} holds. Applying
the inverse $\sigma^{-1}(y) = \log(y/(1-y))$ (the logit transformation)
to both sides,
\[
\lambda S_i(\bx) + \lambda b
\;=\; \log\!\frac{\fprod(\bx)}{1-\fprod(\bx)}
\qquad\forall\bx.
\]
The LHS is affine in $\bx$. Using $u_j$ as in \eqref{eq:u_def} and
$\fprod$ as in \eqref{eq:f_prod_in_u},
\[
\log\frac{\fprod}{1-\fprod}
\;=\; \log e^{U_i} - \log\!\Bigl[\,\prod_j(1+e^{u_j}) - e^{U_i}\Bigr]
\;=\; U_i - \log\Bigl[\,\sum_{T\subsetneq [m_i]} e^{u_T}\Bigr].
\]
Since $\lambda S_i = U_i$, the relation reduces to the equation
\begin{equation}
\lambda b \;=\; -\log\Bigl[\,\sum_{T\subsetneq[m_i]} e^{u_T}\Bigr]
\qquad\forall\,\bu\in\R^{m_i}.
\label{eq:reduced_b}
\end{equation}
For $m_i\ge 2$, the inner sum contains the linearly independent
exponentials $e^{u_\emptyset}=1$ and $e^{u_{\{1\}}}=e^{u_1}$. The
function $\bu\mapsto \log\sum_{T\subsetneq[m_i]} e^{u_T}$ is therefore
a log-sum-exp of at least two affinely independent linear functionals
of $\bu$, hence strictly convex and not constant on
$\R^{m_i}$~\cite{boyd2004convex}. The RHS of \eqref{eq:reduced_b}
thus depends on $\bu$, contradicting the requirement that it equal
the constant $\lambda b$.

For uniqueness of $b$ at the critical point: setting $S_i=0$ in
\eqref{eq:bcw_model} yields $\fwgtbc(\bx_i^{(\mathrm{cp})}) =
\sigma(\lambda b)$, monotone increasing in $b$. The equation
$\sigma(\lambda b) = (1/2)^{m_i}$ has the unique solution
$b = -\lambda^{-1}\log(2^{m_i}-1)$.
\end{proof}

The proof reveals a structural fact: the bias \eqref{eq:bcw_model} is
exactly the value $-\lambda^{-1}\log(\sum_{T\subsetneq[m_i]} e^{u_T})$
evaluated at the critical point $\bu=0$. The state-dependence that
would be required to make this an exact identity has been
\emph{frozen} at its critical-point value. We now make this
mean-field interpretation rigorous.

\begin{proposition}[Bias correction as zero-order mean-field approximation]
\label{prop:mean_field}
The bias-corrected weighted-logistic regulatory function
\eqref{eq:bcw_compact} is the unique single-sigmoid approximation
$\sigma\bigl(\lambda S_i + \lambda b(\bu)\bigr)$ to the
product-of-logistics in which the state-dependent log-sum-exp
correction
\[
\Lambda(\bu) \;:=\; \log\!\Bigl[\,\sum_{T\subsetneq[m_i]} e^{u_T}\Bigr],
\qquad
\sigma^{-1}(\fprod) \;=\; U_i - \Lambda(\bu),
\]
is replaced by its value at the critical point:
\begin{equation}
\Lambda(\bu) \;\mapsto\; \Lambda(\mathbf{0}) \;=\; \log(2^{m_i}-1).
\label{eq:mean_field_replace}
\end{equation}
Equivalently, $\lambda b = -\Lambda(\mathbf{0})$.
\end{proposition}

\begin{proof}
By the same derivation as \eqref{eq:reduced_b}, $\fprod = \sigma(U_i -
\Lambda(\bu))$ exactly. Replacing the convex function $\Lambda(\bu)$
by its constant value at $\bu = \mathbf{0}$ yields
$\sigma(U_i - \Lambda(\mathbf{0})) = \sigma(\lambda S_i -
\log(2^{m_i}-1)) = \fwgtbc$. Conversely, any constant replacement
$\Lambda(\bu)\mapsto c$ producing a single-sigmoid function
$\sigma(\lambda S_i - c)$ that agrees with $\fprod$ at $\bu = \mathbf{0}$
must satisfy $\sigma(-c) = 2^{-m_i}$, i.e.\ $c = \log(2^{m_i}-1) =
\Lambda(\mathbf{0})$.
\end{proof}

\Cref{prop:mean_field} positions the bias correction in a
mean-field/variational-approximation framework: the
product-of-logistics regulatory function has the exact log-odds
representation $\sigma^{-1}(\fprod) = U_i - \Lambda(\bu)$, the
difference of a linear and a strictly convex log-sum-exp functional.
Replacing $\Lambda(\bu)$ by any constant $c$ yields a single-sigmoid
approximation; the choice $c = \Lambda(\mathbf{0})$ minimises the
local error at the critical point and is the unique choice for which
the approximation is exact there. The same construction underlies
mean-field approximations to log-partition functions in statistical
mechanics~\cite{wainwright2008graphical}, in which a state-dependent
correction is approximated by its value at a reference configuration.

A natural follow-up question is whether allowing a more general affine
$\Phi(\bx)$ in a single-sigmoid family $\sigma(\Phi(\bx))$ could
simultaneously match the product's value, slope, \emph{and} curvature
at the critical point. We close this section with a negative answer.

\begin{proposition}[Structural limit of affine single-sigmoid approximations]
\label{prop:affine_limit}
Fix $m_i\ge 2$. Let $\Phi(\bx) = c_0 + \sum_j c_j(x_j-\theta_{ij})$
be an arbitrary affine function. There is no choice of coefficients
$(c_0,c_1,\ldots,c_n)$ for which the single-sigmoid approximation
$F(\bx) = \sigma(\Phi(\bx))$ simultaneously satisfies, at the critical
point $\bx = \bx_i^{(\mathrm{cp})}$:
\begin{equation}
F = \fprod, \qquad
\partial_{x_l}F = \partial_{x_l}\fprod, \qquad
\partial^2_{x_l}F = \partial^2_{x_l}\fprod \quad \forall l\in\Acal_i\cup\Rcal_i.
\label{eq:affine_match}
\end{equation}
\end{proposition}

\begin{proof}
Suppose, for contradiction, that all three conditions hold. From the
value condition $\sigma(c_0) = 1/2^{m_i}$ we obtain $c_0 = -\log(2^{m_i}
-1)\ne 0$ (since $m_i\ge 2$). The slope condition
$\partial_{x_l}F|_{\rm cp} = c_l\sigma'(c_0)\stackrel{!}{=}
\sgn(w_{il})\lambda/2^{m_i+1}$ (using \eqref{eq:slope_prod}) pins down
$c_l = \sgn(w_{il})\lambda/[2^{m_i+1}\sigma'(c_0)]$, which is finite
and nonzero. The curvature condition $\partial^2_{x_l}F|_{\rm cp} =
c_l^2\sigma''(c_0)\stackrel{!}{=}\partial^2_{x_l}\fprod|_{\rm cp} =
0$ by \eqref{eq:fprod_zero_curv}. Since $c_l\ne 0$, this forces
$\sigma''(c_0) = 0$, equivalently $c_0 = 0$, contradicting $c_0 =
-\log(2^{m_i}-1)\ne 0$.
\end{proof}

\Cref{prop:affine_limit} is the structural obstruction underlying
\Cref{prop:curvature_mismatch}: matching the value at the critical
point forces the sigmoid to be evaluated away from its inflection
point, where $\sigma''\ne 0$; matching the slope by tuning the affine
gradient $c_l$ then forces a nonzero curvature
$c_l^2\sigma''\ne 0$, which cannot match $\fprod$'s vanishing
curvature. The bias-corrected weighted formulation \eqref{eq:bcw_model}
corresponds to the choice $c_l = w_{il}\lambda$ (the natural
sign-encoded gradient), in which case both the slope and curvature of
$F$ at the critical point have unequal magnitudes relative to $\fprod$.
A two-parameter family $\sigma(c\lambda S_i + c_0)$ with adjustable
overall scale $c$ can additionally match the slope (by choosing
$c = (2-2^{1-m_i})^{-1}$), but \Cref{prop:affine_limit} shows that
no choice of $(c_0,c)$ can match the curvature as well.

\begin{remark}[Heterogeneous steepness]
\label{rem:nonuniform_lambda}
The construction extends to gene/regulator-specific steepness
$\lambda_{ij}$ at the cost of a more elaborate bias formula. With
$u_j = w_{ij}\,\lambda_{ij}(x_j-\theta_{ij})$, the product-form
critical-point value is still $(1/2)^{m_i}$. Defining
$\tilde{S}_i = \sum_j u_j$, the matching bias becomes
$\tilde{b}_i = -\log(2^{m_i}-1)$, applied additively to $\tilde{S}_i$
rather than to $\lambda S_i$. The slope-ratio formula
\eqref{eq:slope_ratio} extends with $\lambda$ replaced by $\lambda_{il}$
in the $l$-th partial derivative, and \Cref{thm:algebraic_identity}
holds verbatim with the new $u_j$ definition. The numerical
illustrations of this paper assume uniform $\lambda$ for clarity.
\end{remark}

\section{Equivalence of Fixed-Weight and Real-Weight
         Product-of-Logistics}
\label{sec:weight_rescaling}

We now establish a complementary structural fact: within the
product-of-logistics framework, real-valued positive weights are
parameter-equivalent to fixed unit weights. This justifies the
unit-magnitude convention $|w_{ij}|=1$ adopted in
\Cref{rem:magnitude_rescaling} for the analysis of
\Cref{def:weighted_model} and \Cref{def:bcw_model}, and complements
the parallel observation in~\cite[\S 8.3.3]{belgacem2025exploring}
that fixed-weight product-of-logistics enjoys an identifiability
advantage absent from weighted-sum formulations: the rescaling
\eqref{eq:rescale_def} below shows that the apparent freedom of using
real weights $w_{ij} > 0$ in the activator and repressor factors does
not enlarge the dynamics-distinguishable parameter space, so weight
freedom and threshold/steepness freedom collapse into a single
identifiable triple per regulator.

\begin{proposition}[Equivalence under positive parameter rescaling]
\label{prop:weight_rescaling}
Let gene $i$ be regulated by one activator $x_j$ ($w_{ij}>0$) and
one repressor $x_k$ ($w_{ik}>0$). The weighted product-of-logistics
dynamics
\begin{equation}
\dot{x}_i \;=\; \kappa_i\!\left[
   \frac{1}{1+e^{-\lambda_{ij}(w_{ij}x_j - \theta_{ij})}} \cdot
   \frac{1}{1+e^{-\lambda_{ik}(\theta_{ik} - w_{ik}x_k)}}
\right] - \gamma_i x_i
\label{eq:weighted_product}
\end{equation}
admit the parameter-rescaling identity
\begin{equation}
\dot{x}_i \;=\; \kappa_i\!\left[
   \frac{1}{1+e^{-\lambda_{ij}'(x_j - \theta_{ij}')}} \cdot
   \frac{1}{1+e^{-\lambda_{ik}'(\theta_{ik}' - x_k)}}
\right] - \gamma_i x_i,
\label{eq:weighted_product_rescaled}
\end{equation}
under
\begin{equation}
\lambda_{ij}' = \lambda_{ij}\,w_{ij}>0,\quad
\theta_{ij}' = \theta_{ij}/w_{ij}>0,\quad
\lambda_{ik}' = \lambda_{ik}\,w_{ik}>0,\quad
\theta_{ik}' = \theta_{ik}/w_{ik}>0.
\label{eq:rescale_def}
\end{equation}
The two parameterisations produce identical trajectories.
\end{proposition}

\begin{proof}
For the activator factor, factor out $w_{ij}$:
$-\lambda_{ij}(w_{ij}x_j - \theta_{ij}) = -\lambda_{ij}w_{ij}(x_j -
\theta_{ij}/w_{ij}) = -\lambda_{ij}'(x_j-\theta_{ij}')$. The repressor
factor is symmetric. Strict positivity of all rescaled parameters is
immediate from $w_{ij},w_{ik}>0$.
\end{proof}

\begin{remark}[Identifiability consequences]
\label{rem:identifiability}
\Cref{prop:weight_rescaling} reveals an intrinsic identifiability
limit of the weighted product-of-logistics: the products
$\lambda_{ij}\cdot w_{ij}$ and ratios $\theta_{ij}/w_{ij}$ are
recoverable from input-output data, but the individual triples
$(\lambda_{ij},\theta_{ij},w_{ij})$ are not. This is independent of
the data quality; it is a property of the model class. Modelling can
therefore be carried out unambiguously with unit weights and
gene/regulator-specific $(\lambda_{ij},\theta_{ij})$, or with
real positive weights and rescaled parameters, depending on convenience.
The Samuilik formulation \eqref{eq:samuilik_def} does not admit any
such rescaling: the signed weight $w_{ij}^{\mathrm{S}}$ must remain
inside the increasing sigmoid for direction encoding, and any change
of variable that simplifies the activator part propagates to the
repressor part with the wrong sign.
\end{remark}

\section{Equilibrium Correspondence and Linearised Dynamics}
\label{sec:equilibrium}

Theorems~\ref{thm:three_point}--\ref{thm:algebraic_identity} establish
that the regulatory functions $\fprod$ and $\fwgtbc$ coincide as static
maps at the critical-point reference configurations. We now turn to
the corresponding dynamical question: \emph{when do the
\textbf{full ODE systems}
$\dot{x}_i = \kappa_i f_i - \gamma_i x_i$ share equilibria, and how do
their linearised dynamics differ?}

\subsection{Shared critical-point equilibria}
\label{subsec:shared_eq}

\begin{theorem}[Shared critical-point equilibrium]
\label{thm:shared_equilibrium}
Consider the product-of-logistics system
\begin{equation}
\dot{x}_i \;=\; \kappa_i\,\fprod(\bx) - \gamma_i x_i, \qquad i=1,\ldots,n,
\label{eq:dyn_prod}
\end{equation}
and the bias-corrected weighted system
\begin{equation}
\dot{x}_i \;=\; \kappa_i\,\fwgtbc(\bx) - \gamma_i x_i, \qquad i=1,\ldots,n.
\label{eq:dyn_bcw}
\end{equation}
Suppose the parameters satisfy the \emph{critical-point self-consistency
condition}
\begin{equation}
\theta_{ij} \;=\; \frac{\kappa_j}{\gamma_j\,2^{m_j}}
\qquad\text{for every pair }(i,j)\text{ with } w_{ij}\ne 0.
\label{eq:self_consistency}
\end{equation}
Then the configuration
\begin{equation}
\bx^{*} \;:=\; \Bigl(\,\tfrac{\kappa_1}{\gamma_1 2^{m_1}},\ldots,
                       \tfrac{\kappa_n}{\gamma_n 2^{m_n}}\,\Bigr)
\label{eq:shared_eq}
\end{equation}
is a simultaneous equilibrium of \eqref{eq:dyn_prod} and
\eqref{eq:dyn_bcw}, and at $\bx^*$ both regulatory functions evaluate
to $1/2^{m_i}$.
\end{theorem}

\begin{proof}
Under \eqref{eq:self_consistency}, for every $i$ and every regulator
$j$ of $i$, $x_j^* = \kappa_j/(\gamma_j 2^{m_j}) = \theta_{ij}$. Hence
$\bx^*$ is the critical point of every gene, and by
\Cref{thm:three_point}(a), $\fprod(\bx^*) = \fwgtbc(\bx^*) = 1/2^{m_i}$.
The equilibrium equation reads
$\kappa_i\cdot 2^{-m_i} - \gamma_i x_i^* = 0$, which is satisfied by
$x_i^* = \kappa_i/(\gamma_i 2^{m_i})$.
\end{proof}

The self-consistency condition \eqref{eq:self_consistency} is a
non-vacuous but biologically interpretable structural constraint:
the threshold of regulator $j$ as seen by gene $i$ must equal the
critical-point equilibrium concentration of regulator $j$ itself.
Generically the constraint singles out a specific submanifold of the
parameter space $(\kappa_i,\gamma_i,\theta_{ij})$; on this submanifold,
both formulations share a canonical equilibrium with regulatory output
$1/2^{m_i}$. Off this submanifold, the two formulations
generically have \emph{distinct} equilibria, as illustrated
numerically in \Cref{sec:application}.

\begin{corollary}[Globally unique shared equilibrium in the contractive regime]
\label{cor:global_unique}
Assume that, in addition to the self-consistency condition
\eqref{eq:self_consistency}, the parameters satisfy the contraction
condition
\begin{equation}
\gamma_i \;>\; \kappa_i\,\cdot\,\frac{\lambda\,m_i}{4}
\qquad \text{for all } i=1,\ldots,n.
\label{eq:contraction_cond}
\end{equation}
Set
\begin{equation}
\alpha \;:=\;
\min_{i=1,\ldots,n}\!\Bigl(\gamma_i - \kappa_i\,\tfrac{\lambda\,m_i}{4}\Bigr)
\;>\;0.
\label{eq:contraction_rate}
\end{equation}
Then both the product-of-logistics system \eqref{eq:dyn_prod} and the
bias-corrected weighted system \eqref{eq:dyn_bcw} are globally
$\ell^\infty$-contractive on $\R^n_{\ge 0}$ at rate $\alpha$: for any
two solutions $\bx(\cdot),\bx'(\cdot)$ of either formulation,
\begin{equation}
\|\bx(t)-\bx'(t)\|_\infty
\;\le\; e^{-\alpha\,t}\,\|\bx(0)-\bx'(0)\|_\infty
\qquad\text{for all } t\ge 0.
\label{eq:exponential_contraction}
\end{equation}
Each formulation admits the configuration $\bx^{*}$ of
\eqref{eq:shared_eq} as the \emph{unique globally exponentially
attracting} equilibrium, with explicit convergence bound
\begin{equation}
\|\bx(t) - \bx^{*}\|_\infty \;\le\;
e^{-\alpha\,t}\,\|\bx(0) - \bx^{*}\|_\infty,
\label{eq:exp_convergence_eq}
\end{equation}
and from any common initial condition $\bx(0)\in\R^n_{\ge 0}$ the two
trajectories $\bx^{\rm prod}(t)$, $\bx^{\rm bcw}(t)$ both converge to
$\bx^{*}$ at the same guaranteed rate $\alpha$.
\end{corollary}

\begin{proof}
For either formulation, the regulatory function $f_i$ has
component-wise sigmoidal Lipschitz constant $|\partial_{x_l}f_i| \le
\lambda/4$, so the row-sum of $|\partial F_i/\partial x_l|$ at any
point is at most $\kappa_i\,\lambda\,m_i/4$ on the off-diagonal,
combined with the diagonal contribution $-\gamma_i$. Condition
\eqref{eq:contraction_cond} implies that the logarithmic
$\ell^\infty$ norm of the Jacobian satisfies
\[
\mu_{\infty}(J(\bx))
\;\le\; \max_{i=1,\ldots,n}\!\Bigl(-\gamma_i + \kappa_i\,\tfrac{\lambda m_i}{4}\Bigr)
\;=\; -\,\alpha
\;<\;0,
\]
with $\alpha$ given by \eqref{eq:contraction_rate} and positive by
\eqref{eq:contraction_cond}. Standard contraction theory (see, e.g.,
\cite{hirsch2005monotone,belgacem2014mathematical,
belgacem2013stability,belgacem2018reduction}) then yields the
contraction estimate \eqref{eq:exponential_contraction} for any two
solutions of either formulation, and a unique globally exponentially
stable equilibrium for each. By \Cref{thm:shared_equilibrium},
$\bx^{*}$ is an equilibrium of both systems; uniqueness forces it to
be the unique global one for each, and applying
\eqref{eq:exponential_contraction} with $\bx'(\cdot)\equiv\bx^{*}$
yields \eqref{eq:exp_convergence_eq}.
\end{proof}

\Cref{cor:global_unique} is the strongest dynamical equivalence
statement that can be made between the two formulations: in the
contractive regime under self-consistency, not only do the systems
share a critical-point equilibrium but they both converge globally
and exponentially to it from any initial condition. The contraction
condition \eqref{eq:contraction_cond} is the standard
diagonal-dominance criterion for logistic-based GRNs developed
in~\cite{belgacem2014mathematical,belgacem2018reduction,
belgacem2014stability,belgacem2013stability,belgacem2013analysis} for
stability and model-reduction analyses of transcription--translation
systems, and reflects the analytical advantage of working with
logistic functions whose individual Lipschitz constants are bounded
uniformly by $\lambda/4$.

\begin{proposition}[Robustness of the equilibrium correspondence]
\label{prop:robustness_eq}
Let $\boldsymbol\theta^{0} = (\theta_{ij}^0)$ with $\theta_{ij}^0 :=
\kappa_j/(\gamma_j 2^{m_j})$ denote the nominal self-consistent
thresholds of \Cref{thm:shared_equilibrium}, and let
$\bx^{*}$ be the corresponding shared equilibrium \eqref{eq:shared_eq}.
Suppose the actual thresholds satisfy $\theta_{ij} = \theta_{ij}^0 +
\eta_{ij}$ with perturbation matrix
$\boldsymbol\eta = (\eta_{ij})$. Assume the Jacobians
$J^{\rm prod}, J^{\rm bcw}$ at $(\bx^{*}, \boldsymbol\theta^0)$ are
non-singular. Then, for sufficiently small $\|\boldsymbol\eta\|$,
each system has a unique nearby equilibrium $\tilde\bx^{\rm prod}
(\boldsymbol\eta), \tilde\bx^{\rm bcw}(\boldsymbol\eta)$ depending
smoothly on $\boldsymbol\eta$, with leading-order expansions
\begin{align}
\tilde\bx^{\rm prod}(\boldsymbol\eta) - \bx^{*}
   &= (J^{\rm prod})^{-1}\,\Psi^{\rm prod}\boldsymbol\eta
     + O(\|\boldsymbol\eta\|^2),
\label{eq:eq_dev_prod}\\[2pt]
\tilde\bx^{\rm bcw}(\boldsymbol\eta) - \bx^{*}
   &= (J^{\rm bcw})^{-1}\,\Psi^{\rm bcw}\boldsymbol\eta
     + O(\|\boldsymbol\eta\|^2),
\label{eq:eq_dev_bcw}
\end{align}
where $\Psi^{\rm prod}, \Psi^{\rm bcw}$ are the parameter-Jacobian
matrices defined component-wise by
\begin{equation}
(\Psi^{\rm prod})_{i,(jk)} \,=\, \frac{w_{ik}\kappa_i\lambda\,
\mathbf{1}_{\{j=i\}}}{2^{m_i+1}},
\qquad
(\Psi^{\rm bcw})_{i,(jk)} \,=\, \frac{w_{ik}\kappa_i\lambda
(2^{m_i}-1)\,\mathbf{1}_{\{j=i\}}}{4^{m_i}},
\label{eq:psi_def}
\end{equation}
where the index $(jk)$ enumerates the parameter $\theta_{jk}$
(threshold of gene $j$ associated with regulator $k$); the indicator
$\mathbf{1}_{\{j=i\}}$ enforces that only perturbations to the thresholds of
gene $i$ itself contribute to row $i$ of $\Psi$. The
parameter-Jacobian matrices satisfy the rescaling identity
\begin{equation}
\Psi^{\rm bcw} \;=\; D\,\Psi^{\rm prod},
\qquad D = \operatorname{diag}(2-2^{1-m_i}),
\label{eq:psi_rescale}
\end{equation}
so the relative deviation magnitudes satisfy
\begin{equation}
\|\tilde\bx^{\rm bcw}-\bx^{*}\| \;\le\; \|(J^{\rm bcw})^{-1}\|\,
\|D\|\,\|\Psi^{\rm prod}\|\,\|\boldsymbol\eta\|
+ O(\|\boldsymbol\eta\|^2),
\label{eq:eq_dev_bound}
\end{equation}
with $\|D\|\le 2$.
\end{proposition}

\begin{proof}
Define $F^{\rm prod}(\bx,\boldsymbol\theta) := \boldsymbol\kappa\odot
\mathbf{f}^{\rm prod}(\bx;\boldsymbol\theta) - \Gamma\bx$ and
analogously for $F^{\rm bcw}$. Both vanish at
$(\bx^{*}, \boldsymbol\theta^0)$ and have non-singular $\bx$-Jacobian
by assumption, so the implicit function theorem applies, giving
\[
\tilde\bx(\boldsymbol\theta) - \bx^* \;=\; -\Bigl(\frac{\partial F}
{\partial\bx}\Bigr)^{\!-1} \frac{\partial F}{\partial\boldsymbol\theta}
\Bigr|_{(\bx^*,\boldsymbol\theta^0)}(\boldsymbol\theta-\boldsymbol\theta^0)
+ O(\|\boldsymbol\theta-\boldsymbol\theta^0\|^2).
\]
Setting $\Psi := -\partial F/\partial\boldsymbol\theta$ at
$(\bx^*,\boldsymbol\theta^0)$ recovers \eqref{eq:eq_dev_prod} and
\eqref{eq:eq_dev_bcw}. For the component formulas, observe that
$f^{\rm prod}_i$ and $f^{\rm bcw}_i$ depend on $\theta_{jk}$ only when
$j = i$ (i.e.\ on the thresholds of gene $i$ itself), which produces
the indicator $\mathbf{1}_{\{j=i\}}$ in \eqref{eq:psi_def}. For $j = i$ and
$k$ a regulator of $i$, $u_k = w_{ik}\lambda(x_k-\theta_{ik})$, so
$\partial u_k/\partial\theta_{ik} = -w_{ik}\lambda$. At the critical
point this gives $\partial f^{\rm prod}_i/\partial\theta_{ik}|_{\rm cp}
= -w_{ik}\lambda/2^{m_i+1}$ and $\partial f^{\rm bcw}_i/
\partial\theta_{ik}|_{\rm cp} = -w_{ik}\lambda(2^{m_i}-1)/4^{m_i}$.
Multiplying by $\kappa_i$ and absorbing the overall sign into $\Psi$
yields \eqref{eq:psi_def}. The rescaling identity
\eqref{eq:psi_rescale} follows directly from \Cref{prop:slope_ratio},
and the operator-norm bound~\eqref{eq:eq_dev_bound} is standard.
\end{proof}

\Cref{prop:robustness_eq} quantifies the sensitivity of the shared
equilibrium to threshold perturbations. The structural
identity \eqref{eq:psi_rescale} reflects the row-wise rescaling we
already observed for the Jacobian: the parameter-sensitivity matrix
of the bias-corrected weighted formulation is the
$D$-rescaling of the corresponding product-of-logistics matrix.
Operationally, this means that any parameter mis-specification
$\boldsymbol\eta$ propagates to a $1$-to-$(2-2^{1-m_i})$ times
larger displacement of the equilibrium in the bias-corrected
formulation per row.

\subsection{Linearised dynamics at a shared equilibrium}
\label{subsec:linearisation}

The shared equilibrium of \Cref{thm:shared_equilibrium} is a point at
which the regulatory functions agree, but the local dynamics depend
on the \emph{slopes} of the regulatory functions, which differ by the
factor $2-2^{1-m_i}$ established in \Cref{prop:slope_ratio}. The
resulting Jacobians are related by an explicit transformation.

\begin{proposition}[Linearisation comparison at a shared equilibrium]
\label{prop:linearisation}
Let $\bx^*$ be a shared critical-point equilibrium as in
\Cref{thm:shared_equilibrium}. Define the diagonal matrix
\begin{equation}
D \;:=\; \operatorname{diag}\bigl(2-2^{1-m_1},\ldots,
                                  2-2^{1-m_n}\bigr)\in\R^{n\times n},
\label{eq:D_diag}
\end{equation}
and the degradation matrix $\Gamma := \operatorname{diag}
(\gamma_1,\ldots,\gamma_n)$. Let $J^{\mathrm{prod}}, J^{\mathrm{bcw}}
\in\R^{n\times n}$ denote the Jacobians of the right-hand sides of
\eqref{eq:dyn_prod} and \eqref{eq:dyn_bcw} at $\bx^*$. Then
\begin{equation}
\boxed{\;\;
J^{\mathrm{bcw}} \;=\; D\bigl(J^{\mathrm{prod}} + \Gamma\bigr) - \Gamma.
\;\;}
\label{eq:jacobian_relation}
\end{equation}
In particular,
\begin{equation}
J^{\mathrm{bcw}}_{il} \;=\; (2-2^{1-m_i})\,J^{\mathrm{prod}}_{il}
\qquad\text{for }i\ne l,
\label{eq:jacobian_offdiag}
\end{equation}
and the off-diagonal entries are uniformly amplified by the row-wise
factor $2-2^{1-m_i}\in[1,2)$, while the diagonal entries are
transformed according to \eqref{eq:jacobian_relation}.
\end{proposition}

\begin{proof}
Writing $\mathbf{F}^{\rm bcw}(\bx) = \boldsymbol{\kappa}\odot
\mathbf{f}^{\rm wgt,bc}(\bx) - \Gamma\bx$ (with $\odot$ denoting
component-wise product) and similarly for the product system, we have
$J^{\rm bcw}_{il} = \kappa_i\partial_l\fwgtbc(\bx^*) - \gamma_i\delta_{il}$.
By \Cref{prop:slope_ratio}, at the critical point
$\partial_l\fwgtbc(\bx^*) = (2-2^{1-m_i})\,\partial_l\fprod(\bx^*)$,
so $\kappa_i\partial_l\fwgtbc = (2-2^{1-m_i})\kappa_i\partial_l\fprod
= (2-2^{1-m_i})(J^{\rm prod}_{il} + \gamma_i\delta_{il})$. Hence
\[
J^{\rm bcw}_{il}
\;=\; (2-2^{1-m_i})(J^{\rm prod}_{il}+\gamma_i\delta_{il})
       \;-\;\gamma_i\delta_{il},
\]
which is the $(i,l)$-entry of $D(J^{\rm prod}+\Gamma)-\Gamma$. The
off-diagonal case $i\ne l$ has $\delta_{il}=0$, recovering
\eqref{eq:jacobian_offdiag}.
\end{proof}

\Cref{prop:linearisation} has direct dynamical consequences. The map
$M \mapsto D(M + \Gamma) - \Gamma$ preserves the diagonal $-\Gamma$
contribution of degradation but rescales the gain-of-regulation part
of the Jacobian by row-wise factors $2-2^{1-m_i}\in[1,2)$. In
particular:
\begin{enumerate}\itemsep2pt
\item \textbf{Trace}: 
$\operatorname{tr}(J^{\rm bcw}) - \operatorname{tr}(J^{\rm prod})
= \sum_i(2-2^{1-m_i}-1)\bigl(J^{\rm prod}_{ii}+\gamma_i\bigr)
= \sum_i(1-2^{1-m_i})\,\kappa_i\partial_i\fprod(\bx^*)$.
The trace changes only through self-regulation contributions
($w_{ii}\ne 0$); for networks without self-regulation,
$\operatorname{tr}(J^{\rm bcw}) = \operatorname{tr}(J^{\rm prod})$
and the spectral barycentre is preserved.
\item \textbf{Off-diagonal amplification}: every off-diagonal coupling
is amplified by the row-wise factor $2-2^{1-m_i}$. Bounded above by
$2$, this amplification can shift eigenvalues either by enhancing
oscillation frequencies (cooperative crossings) or by destabilising
saddles (when off-diagonals control the sign of subprincipal minors).
\item \textbf{Symmetric homogeneous case}: if all genes have the same
regulator count $m_i = m$, then $D = (2-2^{1-m})I$ and
\eqref{eq:jacobian_relation} reduces to
$J^{\rm bcw} = (2-2^{1-m})J^{\rm prod} + (1-2^{1-m})\Gamma$. For
the standard case $\gamma_i = \gamma$ for all $i$, the eigenvalues of
$J^{\rm bcw}$ are an affine transform of those of $J^{\rm prod}$:
\begin{equation}
\mu^{\rm bcw} \;=\; (2-2^{1-m})\,\mu^{\rm prod} + (1-2^{1-m})\gamma.
\label{eq:eigenvalue_map}
\end{equation}
Equivalently, $\mu^{\rm bcw} + \gamma = (2-2^{1-m})(\mu^{\rm prod}+\gamma)$.
\end{enumerate}

The identity \eqref{eq:eigenvalue_map} has a striking consequence: the
affine map $\mu^{\rm prod}\mapsto\mu^{\rm bcw}$ has the unique real
fixed point $\mu = -\gamma$. Modes of $J^{\rm prod}$ with $\operatorname{Re}\mu^{\rm prod}= -\gamma$
are therefore mapped to modes of $J^{\rm bcw}$ with $\operatorname{Re}\mu^{\rm bcw} =
-\gamma$, irrespective of their imaginary parts.

\begin{corollary}[Stability dichotomy at shared equilibria]
\label{cor:stability_dichotomy}
In the symmetric homogeneous case ($m_i = m$, $\gamma_i = \gamma$ for
all $i$), the local stability margin at a shared equilibrium changes
under the bias correction according to the formula
\begin{equation}
\operatorname{Re}\mu^{\rm bcw} - \operatorname{Re}\mu^{\rm prod}
\;=\; (1-2^{1-m})\,\bigl(\operatorname{Re}\mu^{\rm prod} + \gamma\bigr).
\label{eq:stability_shift}
\end{equation}
Consequently, each eigenvalue $\mu^{\rm prod}$ satisfies:
\begin{enumerate}\itemsep1pt
\item[(a)] $\operatorname{Re}\mu^{\rm prod} > -\gamma$
            $\;\Longrightarrow\;$
           $\operatorname{Re}\mu^{\rm bcw} > \operatorname{Re}\mu^{\rm prod}$
           (\emph{the bias correction destabilises slowly-decaying
           or unstable modes});
\item[(b)] $\operatorname{Re}\mu^{\rm prod} < -\gamma$
            $\;\Longrightarrow\;$
           $\operatorname{Re}\mu^{\rm bcw} < \operatorname{Re}\mu^{\rm prod}$
           (\emph{the bias correction further stabilises fast-decaying
           modes});
\item[(c)] $\operatorname{Re}\mu^{\rm prod} = -\gamma$
            $\;\Longrightarrow\;$
           $\operatorname{Re}\mu^{\rm bcw} = -\gamma$
           (\emph{the fixed-point real part is preserved exactly}).
\end{enumerate}
The threshold $\operatorname{Re}\mu^{\rm prod} = -\gamma$ corresponds
to the case where the regulation part of the Jacobian contributes a
purely imaginary spectrum to $\mu^{\rm prod}$. For networks without
self-regulation, $J^{\rm prod} = -\gamma I +
(\kappa\lambda/2^{m+1})\,A$ with $A$ the signed adjacency matrix, and
this case arises whenever $A$ has purely imaginary eigenvalues---a
common situation on negative-feedback cyclic motifs.
\end{corollary}

\begin{proof}
Take real parts on both sides of \eqref{eq:eigenvalue_map}:
$\operatorname{Re}\mu^{\rm bcw} = (2-2^{1-m})\operatorname{Re}\mu^{\rm prod}
+ (1-2^{1-m})\gamma$. Subtract $\operatorname{Re}\mu^{\rm prod}$:
$\operatorname{Re}\mu^{\rm bcw} - \operatorname{Re}\mu^{\rm prod} =
(1-2^{1-m})\operatorname{Re}\mu^{\rm prod} + (1-2^{1-m})\gamma =
(1-2^{1-m})(\operatorname{Re}\mu^{\rm prod}+\gamma)$. The factor
$(1-2^{1-m})$ is strictly positive for $m\ge 2$, so the sign of the
shift coincides with the sign of $\operatorname{Re}\mu^{\rm prod}+\gamma$.
\end{proof}

\Cref{cor:stability_dichotomy} provides an actionable design
implication. In a network model where the product-of-logistics
Jacobian has eigenvalues clustered near $\operatorname{Re}\mu = -\gamma$
(the natural decay rate of an isolated gene), the bias-corrected
weighted formulation preserves the stability margin. Modes that decay
slower than $-\gamma$ (e.g.\ due to strong positive feedback) are
\emph{magnified} by the bias correction and may bifurcate to
instability that is not present in the product. Conversely, fast-decaying
modes are made even faster. This anisotropic
re-scaling of the spectrum is qualitatively different from a uniform
gain modification.

\begin{example}[3-node repressilator-with-activator motif]
\label{ex:3node_eigenvalues}
Consider three genes in a cycle, each with one cross-activator and one
cross-repressor ($m_i = 2$ for all $i$), and no self-regulation:
gene $1$ is activated by gene $2$ and repressed by gene $3$, gene $2$
by gene $3$ and gene $1$ respectively, and gene $3$ by gene $1$ and
gene $2$ respectively. Set $\kappa_i = \kappa$ and $\gamma_i = \gamma$
for all $i$. The self-consistency condition
\eqref{eq:self_consistency} reduces to $\theta_{ij} = \kappa/(4\gamma)$
for every interacting pair $(i,j)$. At the shared equilibrium
$x_i^* = \kappa/(4\gamma)$, the Jacobian of the product-of-logistics
is
\[
J^{\mathrm{prod}} \;=\; -\gamma I + \frac{\kappa\lambda}{8}\,A,
\]
where $A$ is the adjacency matrix with $A_{12}=A_{23}=A_{31}=+1$
(activations) and $A_{13}=A_{21}=A_{32}=-1$ (repressions). The matrix
$A$ has eigenvalues $0,\pm i\sqrt{3}$, so the eigenvalues of
$J^{\mathrm{prod}}$ are $-\gamma$ and $-\gamma\pm i\sqrt{3}\kappa\lambda/8$.
By \eqref{eq:eigenvalue_map} with $m=2$,
\[
\mu^{\mathrm{bcw}} \;=\; \tfrac{3}{2}\mu^{\mathrm{prod}} +
                         \tfrac{1}{2}\gamma,
\]
giving $\mu^{\mathrm{bcw}} = -\gamma$ for the real eigenvalue and
$\mu^{\mathrm{bcw}} = -\gamma\pm i\,(3\sqrt{3}/16)\kappa\lambda$ for
the complex pair. All eigenvalues of $J^{\rm prod}$ in this motif lie
on the line $\operatorname{Re}\mu = -\gamma$, so by
\Cref{cor:stability_dichotomy}(c) the bias correction preserves their
real parts exactly; the only effect of the bias is to amplify the
imaginary parts by the factor $3/2 = 2-2^{1-m}$, increasing the local
oscillation frequency at the shared equilibrium by $50\%$ in the
bias-corrected formulation.
\end{example}

\Cref{ex:3node_eigenvalues} illustrates a practical caveat in moving
between formulations: even when the equilibrium location is preserved
(via the self-consistency condition), the linearised dynamics around
that equilibrium can differ by sizeable factors. A linear feedback
controller designed for one formulation will, in general, exhibit
modified loop gain when applied to the other. The factor
$2-2^{1-m_i}$ from \Cref{prop:slope_ratio} provides an explicit
correction, and the stability-dichotomy formula
\eqref{eq:stability_shift} determines whether the correction is
stabilising or destabilising for each spectral mode.

\section{Application: A Two-Gene Negative-Feedback Oscillator with Self-Activation}
\label{sec:application}

We illustrate the equilibrium-correspondence and linearisation results
of \Cref{sec:equilibrium} on a canonical 2-gene regulatory motif: a
negative feedback loop with positive autoregulation, exhibiting
sustained oscillations via Hopf bifurcation. This motif is among the
simplest gene-circuit topologies known to produce limit-cycle
behaviour without delays, and its symmetric structure makes the
spectral predictions of \Cref{prop:linearisation,cor:stability_dichotomy}
fully explicit.

This setting complements three prior analyses on related two-gene
networks. The Vinoth \emph{et~al.}~\cite{vinoth2025extreme} model is a
delay-coupled DDE in which the topology is \emph{mutual cross-activation
plus cooperative self-repression with delays}: the genes are
$A\!\to\!B$ and $B\!\to\!A$ at the activation level and $A\!\dashv\!A$,
$B\!\dashv\!B$ at the repression level, and oscillations arise via a
delay-induced Hopf bifurcation. Belgacem~\cite{belgacem2026beyond}
analyses the same Vinoth network under two logistic reformulations and
derives the closed-form parameter expressions $\kappa_1 = 4 g_A$,
$\theta_B = g_A$, $\lambda_1 = \ln 3/g_A$ for the weighted-logistic
formulation, with critical self-repression delay
$\tau_c \approx 1.64$ min. The factors $4 = 2^{m_i}$ and
$\ln 3 = \log(2^{m_i}-1)$ in those parameters are
\emph{not coincidences}: they coincide with the saturation
exponent and the bias-correction term of the present framework
(\Cref{def:bcw_model}) at $m_i = 2$, so Belgacem's weighted-logistic
formulation is precisely the bias-corrected weighted formulation of
this paper applied to the Vinoth DDE topology, with parameters fixed
by basal-rate and slope-matching against the Hill-linear baseline.
Belgacem~\cite{belgacem2026sustained} further establishes sustained
limit cycles in the logistic two-gene oscillator through a
delay-driven Hopf-bifurcation analysis, with explicit critical-delay
characterisations and Lyapunov-coefficient computations confirming
the bifurcation criticality; that result, complementary to the
present delay-free analysis, illustrates how the bcw framework
extends naturally into the delay-coupled regime.

In contrast, the present \Cref{sec:application} uses a different
topology --- \emph{mutual self-activation plus asymmetric cross-regulation
without delays} (specified below) --- chosen so that the Hopf
bifurcation is parameter-induced rather than delay-induced and the
characteristic polynomial reduces to a $2\times 2$ algebraic problem
without transcendental delay terms. The three settings are
therefore complementary: Belgacem's two prior analyses
\cite{belgacem2026beyond,belgacem2026sustained} specialise the bcw
framework to the Vinoth DDE topology and derive delay-dependent Hopf
curves and sustained limit-cycle existence; the analysis below
specialises it to a delay-free oscillator and derives a clean
algebraic eigenvalue map that makes the $D = 3/2$ frequency
amplification and the $2/3$ Hopf-threshold ratio transparent.

\subsection{Two-gene oscillator: topology and self-consistent thresholds}
\label{subsec:hopf_motif}

Consider two genes with the regulatory wiring
\[
w_{AA} = +1,\quad w_{AB} = -1,\qquad
w_{BB} = +1,\quad w_{BA} = +1,
\]
so each gene self-activates ($w_{ii}>0$), gene $A$ is repressed by
gene $B$ ($w_{AB}<0$), and gene $B$ is activated by gene $A$
($w_{BA}>0$). Each gene therefore has $m_i = 2$ regulators
(one self-regulator plus one cross-regulator). The closed feedback
loop $A\!\to\!B\!\dashv\!A$ is negative; combined with the local
positive autoregulation, the architecture is a textbook
positive-plus-negative feedback motif of the kind underlying many
biological clocks~\cite{elowitz2000synthetic,gardner2000construction}.

\paragraph*{Threshold structure under self-consistency.}
The motif involves four threshold parameters, $\theta_{AA}, \theta_{AB},
\theta_{BA}, \theta_{BB}$, with $\theta_{ij}$ interpreted biologically
as the concentration of regulator $j$ at which it exerts half-maximal
effect on the regulation of gene $i$ (an $\mathrm{EC}_{50}$ for
activation or $\mathrm{IC}_{50}$ for repression). \emph{In principle
these four thresholds are independent intrinsic biological quantities}
(binding affinities $K_d$, etc.) and need not coincide with any
equilibrium concentration. We \emph{choose} them, however, to
satisfy the self-consistency condition of
\Cref{thm:shared_equilibrium}, which makes the product-of-logistics
and the bias-corrected weighted formulations share equilibria. The
self-consistency identity is derived as follows: at the
``critical point'' where all regulators sit at their respective
thresholds, the product-of-logistics evaluates to $1/2^{m_i} = 1/4$;
the equilibrium condition $\dot x_i = 0$ at this configuration yields
$x_i = \kappa_i/(\gamma_i\cdot 2^{m_i})$. Imposing that this candidate
equilibrium coincides with the threshold seen by other genes (the
self-consistency identity $\theta_{ij} = \kappa_j/(\gamma_j 2^{m_j})$)
determines a specific parameter choice expressed entirely in terms of
$(\kappa, \gamma, m)$:
\begin{equation}
\theta_A \;:=\; \frac{\kappa_A}{4\gamma_A},
\qquad
\theta_B \;:=\; \frac{\kappa_B}{4\gamma_B},
\label{eq:hopf_thetaAB}
\end{equation}
with $\theta_{AA} = \theta_{BA} = \theta_A$ (both involve regulator
$A$) and $\theta_{AB} = \theta_{BB} = \theta_B$ (both involve regulator
$B$). These are \emph{chosen parameter values}, deterministic
functions of $(\kappa, \gamma)$; they are not derived by solving the
dynamics. Genes $A$ and $B$ are distinct species, so $\theta_A$ and
$\theta_B$ are in general distinct biological scales. We adopt
uniform degradation $\gamma_A = \gamma_B = \gamma$ (necessary for the
clean eigenvalue map $\mu^{\rm bcw} = (3/2)\,\mu^{\rm prod} +
(1/2)\gamma$ of \Cref{prop:linearisation}) but \emph{distinct}
synthesis rates $\kappa_A \ne \kappa_B$, so that the two thresholds
$\theta_A = \kappa_A/(4\gamma)$ and $\theta_B = \kappa_B/(4\gamma)$
take different numerical values reflecting the gene-specific kinetics.
With this parameter choice, the dynamical system admits the
shared critical-point equilibrium
\begin{equation}
\bx^* \;=\; (\theta_A,\,\theta_B) \;=\;
\left(\frac{\kappa_A}{4\gamma},\,\frac{\kappa_B}{4\gamma}\right),
\label{eq:hopf_xstar}
\end{equation}
at which both $\fprod$ and $\fwgtbc$ evaluate to $1/2^{m_i} = 1/4$
(\Cref{thm:shared_equilibrium}). The identity $x_j^{*} = \theta_j$
is then a \emph{consequence} of the parameter choice, not its
definition.

\subsection{The four formulations}
\label{subsec:hopf_four}

The four sigmoidal formulations of \Cref{sec:formulation} produce the
following explicit dynamics on the motif:

\paragraph*{Product-of-logistics.}
\begin{align}
\dot{A} &= \kappa_A\,
   \underbrace{\sigma\!\bigl(\lambda(A-\theta_A)\bigr)}_{\text{self-activation}}
   \cdot\,
   \underbrace{\sigma\!\bigl(\lambda(\theta_B-B)\bigr)}_{\text{repression by }B}
   \;-\;\gamma\,A,
\label{eq:hopf_prod_A}\\
\dot{B} &= \kappa_B\,
   \underbrace{\sigma\!\bigl(\lambda(B-\theta_B)\bigr)}_{\text{self-activation}}
   \cdot\,
   \underbrace{\sigma\!\bigl(\lambda(A-\theta_A)\bigr)}_{\text{activation by }A}
   \;-\;\gamma\,B.
\label{eq:hopf_prod_B}
\end{align}

\paragraph*{Bias-corrected weighted.}
With $S_A = (A-\theta_A) - (B-\theta_B)$ and
$S_B = (B-\theta_B)+(A-\theta_A)$, and $b_i = -(\log 3)/\lambda$,
\begin{align}
\dot{A} \;=\; \kappa_A\,\frac{1}{1 + 3\,e^{-\lambda S_A}}
   \;-\;\gamma\,A,\qquad
\dot{B} \;=\; \kappa_B\,\frac{1}{1 + 3\,e^{-\lambda S_B}}
   \;-\;\gamma\,B.
\label{eq:hopf_bcw}
\end{align}

\paragraph*{Unified weighted-logistic (no bias).}
\begin{align}
\dot{A} \;=\; \kappa_A\,\frac{1}{1 + e^{-\lambda S_A}}
   \;-\;\gamma\,A,\qquad
\dot{B} \;=\; \kappa_B\,\frac{1}{1 + e^{-\lambda S_B}}
   \;-\;\gamma\,B.
\label{eq:hopf_wgt}
\end{align}

\paragraph*{Samuilik weighted-sum.}
The canonical Samuilik threshold prescription gives
$\theta_A^{\mathrm{S}} = \tfrac{1}{2}(w_{AA}^{\mathrm{S}} +
w_{AB}^{\mathrm{S}}) = \tfrac{1}{2}(1-1) = 0$ and
$\theta_B^{\mathrm{S}} = \tfrac{1}{2}(w_{BB}^{\mathrm{S}} +
w_{BA}^{\mathrm{S}}) = \tfrac{1}{2}(1+1) = 1$. Setting the Samuilik
steepness $\mu_i = \lambda$ to align the comparison with the other
three formulations, this gives
\begin{align}
\dot{A} \;=\; \kappa_A\,\frac{1}{1 + e^{-\lambda(A-B)}}
   \;-\;\gamma\,A,\qquad
\dot{B} \;=\; \kappa_B\,\frac{1}{1 + e^{-\lambda(A+B-1)}}
   \;-\;\gamma\,B.
\label{eq:hopf_sam}
\end{align}
The mixed-sign content of gene $A$ ($w^{\mathrm S}_{AA}=+1,
w^{\mathrm S}_{AB}=-1$) collapses $\theta_A^{\mathrm S}$ to zero, an
instance of the pathological threshold-degeneration behaviour
catalogued in \Cref{subsec:mixed}.

\subsection{Equilibrium structure}
\label{subsec:hopf_equilibria}

Under the self-consistency condition \eqref{eq:hopf_thetaAB}, the
product-of-logistics and the bias-corrected weighted formulation
share the canonical equilibrium
$\bx^*_{\rm prod} = \bx^*_{\rm bcw} = (\theta_A, \theta_B) =
(\kappa_A/(4\gamma),\kappa_B/(4\gamma))$, since at this configuration
$S_A = S_B = 0$ and both functions evaluate to $\fprod = \fwgtbc =
1/4$, giving $\kappa_i\cdot\tfrac{1}{4} = \gamma\,\theta_i$ for each
gene $i$.

The unbiased weighted formulation \eqref{eq:hopf_wgt} does not admit
$\bx^*$ as an equilibrium under the same self-consistency
parameters: at the point $S_A = S_B = 0$ (forced by $A = \theta_A,
B = \theta_B$), $\fwgt$ evaluates to $1/2$ rather than $1/4$, hence
$\kappa_i \cdot \tfrac{1}{2} = 2\gamma\theta_i \ne \gamma\theta_i$.
The system therefore settles to an \emph{asymmetric}, non-shared
attractor (\Cref{fig:dynamics}, panel (c), orange dot), typically a
winner-take-all state in which one gene saturates while the other is
repressed.

The Samuilik formulation \eqref{eq:hopf_sam}, with its canonical
thresholds $(\theta_A^{\mathrm{S}},\theta_B^{\mathrm{S}}) = (0,1)$,
admits the equilibrium $\bx^*_{\mathrm{sam}}$ at which both Samuilik
arguments $(A-B)$ and $(A+B-1)$ vanish, giving
$A = B = 1/2$ and $\kappa_i\cdot\tfrac{1}{2} = \gamma\,(1/2)$ — a
configuration consistent only when $\kappa_A = \kappa_B = \gamma$.
For general $(\kappa_A, \kappa_B)$, the Samuilik equilibrium is
determined by the nonlinear system $A = \kappa_A\,\sigma(\lambda
(A-B))/\gamma$, $B = \kappa_B\,\sigma(\lambda(A+B-1))/\gamma$, and
in general differs from both $\bx^*$ and $\bx^*_{\rm wgt}$,
displaying the equilibrium misalignment that distinguishes Samuilik
from the product-of-logistics already at the level of critical
points.

\subsection{Linearisation and Hopf bifurcation}
\label{subsec:hopf_linearisation}

The Jacobian of the product-of-logistics at the shared equilibrium
$\bx^*$ is
\begin{equation}
J^{\mathrm{prod}} \;=\; -\gamma\,I + \frac{\lambda}{8}\,K\,W,
\qquad
K = \begin{pmatrix}\kappa_A & 0 \\ 0 & \kappa_B\end{pmatrix},
\qquad
W = \begin{pmatrix} +1 & -1 \\ +1 & +1 \end{pmatrix},
\end{equation}
since for $m_i = 2$ each partial derivative
$\partial_l \fprod\big|_{\bx^*} = w_{il}\lambda\sigma'(0)\,\sigma(0)
= w_{il}\lambda\cdot(1/4)\cdot(1/2) = w_{il}\lambda/8$
by \Cref{prop:slope_ratio} with one differentiated factor and
$m_i-1=1$ non-differentiated factor evaluated at the critical
point. The matrix
$K W$ has trace $\kappa_A+\kappa_B$ and determinant $2\kappa_A
\kappa_B$, hence (when the asymmetry is moderate, $(3-2\sqrt{2})<
\kappa_A/\kappa_B < 3+2\sqrt{2}$) complex conjugate eigenvalues
$(\kappa_A+\kappa_B)/2 \pm i\sqrt{8\kappa_A\kappa_B -
(\kappa_A+\kappa_B)^2}/2$. The eigenvalues of $J^{\mathrm{prod}}$
are therefore
\begin{equation}
\mu^{\mathrm{prod}}_\pm
\;=\;
-\gamma + \frac{\lambda(\kappa_A+\kappa_B)}{16}
\pm i\,\frac{\lambda\sqrt{8\kappa_A\kappa_B-(\kappa_A+\kappa_B)^2}}{16},
\end{equation}
and the equilibrium $\bx^*$ undergoes a Hopf bifurcation when
$\mathrm{Re}(\mu^{\mathrm{prod}}_\pm) = 0$, that is, at
\begin{equation}
\lambda^{\mathrm{prod}}_{\mathrm{Hopf}}
\;=\; \frac{16\,\gamma}{\kappa_A+\kappa_B}.
\label{eq:hopf_threshold_prod}
\end{equation}
By the eigenvalue map of \Cref{prop:linearisation} (with $m_i = 2$
giving $D = 2 - 2^{1-2} = 3/2$) and uniform $\gamma$, the
bias-corrected weighted Jacobian satisfies
$\mu^{\mathrm{bcw}}_\pm = (3/2)\,\mu^{\mathrm{prod}}_\pm +
\tfrac{1}{2}\gamma$, hence
\begin{equation}
\mu^{\mathrm{bcw}}_\pm
\;=\;
-\gamma + \frac{3\lambda(\kappa_A+\kappa_B)}{32}
\pm i\,\frac{3\lambda\sqrt{8\kappa_A\kappa_B-(\kappa_A+\kappa_B)^2}}{32},
\qquad
\lambda^{\mathrm{bcw}}_{\mathrm{Hopf}}
\;=\; \frac{32\,\gamma}{3(\kappa_A+\kappa_B)}.
\label{eq:hopf_threshold_bcw}
\end{equation}
Three structural consequences follow:
\begin{itemize}\itemsep1pt
\item \textbf{Equilibrium agreement.} Both formulations share
$\bx^* = (\theta_A,\theta_B)$ for every $\lambda$
(\Cref{thm:shared_equilibrium}).
\item \textbf{Frequency amplification.} At every fixed $\lambda$,
the oscillation frequency in the bias-corrected formulation is
amplified by the factor $D = 3/2$ relative to the product:
$\mathrm{Im}(\mu^{\mathrm{bcw}}_\pm) = (3/2)\,
\mathrm{Im}(\mu^{\mathrm{prod}}_\pm)$.
\item \textbf{Earlier instability.} The bias-corrected formulation
bifurcates at $\lambda^{\mathrm{bcw}}_{\mathrm{Hopf}} = (2/3)
\lambda^{\mathrm{prod}}_{\mathrm{Hopf}}$, strictly smaller than the
product threshold. For $\lambda\in
[\lambda^{\mathrm{bcw}}_{\mathrm{Hopf}},
\lambda^{\mathrm{prod}}_{\mathrm{Hopf}})$ the bias-corrected system
oscillates while the product-of-logistics is still globally
exponentially stable: an explicit ``destabilising'' instance of the
stability dichotomy \eqref{eq:stability_shift}.
\end{itemize}

\subsection{Numerical illustration}
\label{subsec:hopf_numerics}

We solve the four ODE systems with parameters $\kappa_A = 1$,
$\kappa_B = 1.1$, $\gamma = 1$, $\lambda = 6.5$, and the
self-consistent thresholds $\theta_A = 0.25,\;\theta_B = 0.275$,
starting from $\bx(0) = (\theta_A+0.01,\,\theta_B-0.01)$. The
moderate asymmetry $\kappa_B/\kappa_A = 1.1$ produces
distinct biological scales $\theta_A\ne\theta_B$ for the two genes
while preserving the limit-cycle basin around the shared equilibrium
$\bx^* = (0.25, 0.275)$. The chosen $\lambda = 6.5$ sits between
the two Hopf thresholds:
$\lambda^{\mathrm{bcw}}_{\mathrm{Hopf}} = 32/(3\cdot 2.1) \approx
5.08 < 6.5 < \lambda^{\mathrm{prod}}_{\mathrm{Hopf}} = 16/2.1
\approx 7.62$. Integration uses the adaptive Adams--BDF method
LSODA via R \texttt{deSolve::ode} on $[0, 300]$ with $30{,}001$ time
points; the transient is the first $200$ time units.

\begin{figure}[t]
\centering
\includegraphics[width=0.99\linewidth]{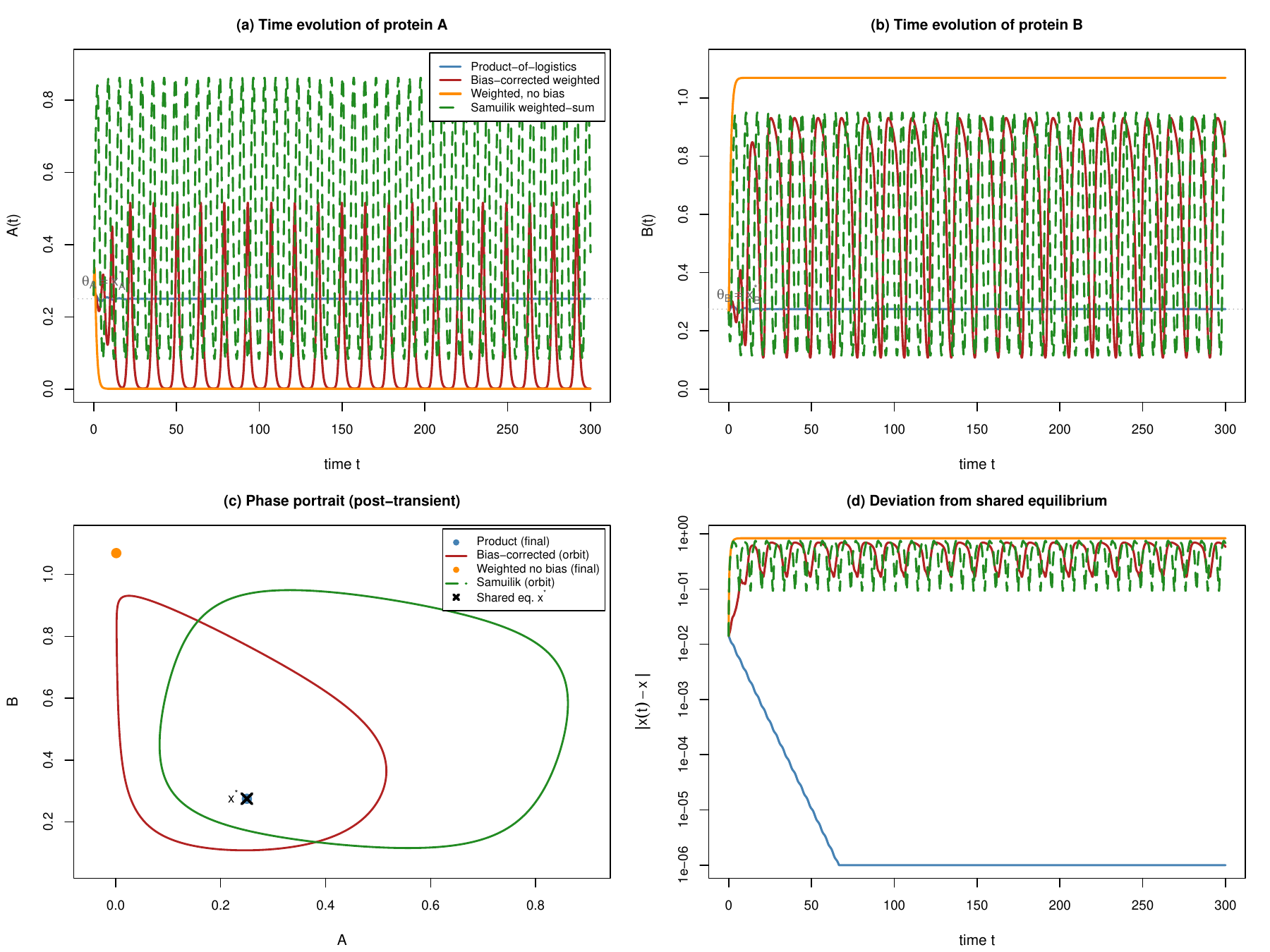}
\caption{Numerical comparison of the four formulations on the
two-gene negative-feedback oscillator at $\lambda = 6.5\in
[\lambda^{\mathrm{bcw}}_{\mathrm{Hopf}},
\lambda^{\mathrm{prod}}_{\mathrm{Hopf}}) = [5.08, 7.62)$ with
asymmetric synthesis $\kappa_A=1,\,\kappa_B=1.1$ and uniform
degradation $\gamma=1$ (distinct thresholds $\theta_A=0.25,\,
\theta_B=0.275$). \textbf{(a,b)} Time evolution of $A$ and $B$. The
product-of-logistics \eqref{eq:hopf_prod_A}--\eqref{eq:hopf_prod_B}
converges to the shared equilibrium $\bx^* = (\theta_A,\theta_B) =
(0.25,\,0.275)$ (stable in this regime); the bias-corrected weighted
formulation \eqref{eq:hopf_bcw} oscillates around the same $\bx^*$
(unstable in this regime); the unbiased weighted formulation
\eqref{eq:hopf_wgt} converges to a non-canonical winner-take-all
state; the Samuilik formulation \eqref{eq:hopf_sam} oscillates
with larger amplitude and around a different mean. \textbf{(c)}
Phase portrait post-transient; the bcw limit cycle and the
product/wgt equilibria are clearly separated. \textbf{(d)}
Deviation $\|\bx(t) - \bx^*\|$ from the shared critical-point
equilibrium $\bx^*$ on a logarithmic scale; only the product
trajectory contracts onto $\bx^*$, while the others depart by
orders of magnitude.}
\label{fig:dynamics}
\end{figure}

\Cref{fig:dynamics} reveals four qualitatively distinct outcomes at
the \emph{same} parameter values:

\begin{itemize}\itemsep2pt
\item The \emph{product-of-logistics} \eqref{eq:hopf_prod_A}--%
\eqref{eq:hopf_prod_B} converges to the canonical equilibrium
$\bx^* = (\theta_A,\theta_B) = (0.25, 0.275)$, consistent with the
linearisation prediction
$\mathrm{Re}(\mu^{\mathrm{prod}}_\pm) = -\gamma +
\lambda(\kappa_A+\kappa_B)/16 = -1 + 6.5\cdot 2.1/16
\approx -0.147 < 0$.
\item The \emph{bias-corrected weighted formulation}
\eqref{eq:hopf_bcw} shares the equilibrium $\bx^*$ but with
$\mathrm{Re}(\mu^{\mathrm{bcw}}_\pm) = -\gamma + 3\lambda
(\kappa_A+\kappa_B)/32 \approx +0.280 > 0$. The trajectory spirals
outward to a limit cycle of approximate range $\approx 0.5$ in
gene-$A$ amplitude.
\item The \emph{unbiased weighted-logistic formulation}
\eqref{eq:hopf_wgt} converges to a non-canonical winner-take-all
equilibrium $\approx (0.001, 1.07)$, far from the shared critical
point. This illustrates the equilibrium mismatch of
\Cref{lem:critical_values}: $\fwgt(\bx^*) = 1/2\ne 1/4 =
\fprod(\bx^*)$.
\item The \emph{Samuilik formulation} \eqref{eq:hopf_sam}
oscillates with larger amplitude around its own (still distinct)
mean --- a limit cycle entirely distinct from the bcw oscillation,
driven by the asymmetric threshold pair
$(\theta_A^{\mathrm S},\theta_B^{\mathrm S}) = (0,1)$ that
the canonical Samuilik prescription forces on the mixed-sign gene.
\end{itemize}

The trajectory deviation panel \Cref{fig:dynamics}(d) summarises
the qualitative ordering: the bias-corrected weighted formulation
shares its equilibrium with the product (the only formulation for
which $\|\bx(t)-\bx^*\|$ remains bounded by an oscillation envelope
around $\bx^*$), while the unbiased and Samuilik alternatives
escape to entirely different attractors.

\subsection{Theoretical predictions confirmed}
\label{subsec:hopf_takeaway}

The Hopf bifurcation analysis makes four theorems of
\Cref{sec:formulation,sec:equilibrium} concrete on a textbook motif.
First, \Cref{prop:bias_canonicity} justifies the specific bias value
$b_i = -\log(2^{m_i}-1)/\lambda = -\log 3/\lambda$ used throughout the
example as the unique single-sigmoid correction that recovers the
product's critical-point value $1/4$ at $m_i=2$, ruling out any other
constant correction. Second, \Cref{thm:shared_equilibrium} forces the
product and the bias-corrected weighted formulation to share the
critical-point equilibrium $\bx^*$ — verified by the numerical
convergence and the linear-stability calculations. Third, the
eigenvalue map of \Cref{prop:linearisation} predicts the exact ratio
$3/2$ between the bcw and product oscillation frequencies near each
formulation's Hopf bifurcation — verified by the closed-form
characteristic equations. Fourth, the stability dichotomy of
\Cref{cor:stability_dichotomy} explains why the bcw formulation
undergoes Hopf bifurcation \emph{earlier} (at smaller $\lambda$) than
the product: the common shift $(1-2^{1-m})\gamma$ destabilises
spectral modes for which $\mathrm{Re}(\mu^{\mathrm{prod}}) > -\gamma$,
which is the case for the complex eigenvalue pair driving the
oscillation. The bias correction therefore acts as a \emph{prediction}
of how the bias-corrected weighted formulation behaves in the parameter
regime where the product is still stable: a structural shift in the
bifurcation diagram, not a mere recalibration of the equilibrium.

\section{Discussion}
\label{sec:discussion}

\subsection{When to use each formulation}
\label{subsec:when_to_use}

The three formulations encode genuinely different modelling
hypotheses, and each is preferable in its own analytical context.

\paragraph*{Product-of-logistics
\eqref{eq:product_of_logistics_intro}.}
The product-of-logistics is the formulation of choice when:
\begin{itemize}\itemsep2pt
\item Multi-regulator interactions are interpreted as multiplicative
occupancy events (independent binding sites, AND-like cooperativity).
\item Bifurcation analysis on canonical motifs (toggle switch,
repressilator, oscillators) is the analytical objective. The
structural symmetry $\sigma''(\theta)=0$ enforced by the product
\eqref{eq:fprod_zero_curv} drives the supercriticality of pitchfork
and Hopf bifurcations through cubic Lyapunov coefficients
(see~\cite{belgacem2025exploring,belgacem2026logistic} for closed-form
results).
\item Each regulator's contribution must be analytically separable for
sensitivity analysis, observer design, or modular extension to
combinatorial logic gates via the De~Morgan product formula.
\end{itemize}

\paragraph*{Bias-corrected weighted-logistic \eqref{eq:bcw_model}.}
The bias-corrected weighted formulation is the formulation of choice
when:
\begin{itemize}\itemsep2pt
\item A single sigmoid per gene is required for tractability of
high-dimensional analysis: linearisation around equilibria, gain
scheduling, model-predictive control synthesis, or large-scale
parameter estimation. Each nonzero off-diagonal entry of the
Jacobian of \eqref{eq:bcw_model} requires \emph{one} sigmoid-derivative
evaluation, whereas the corresponding entry of the
product-of-logistics Jacobian \eqref{eq:f_prod_def} requires one
sigmoid derivative \emph{and} a product over $m_i-1$ remaining
sigmoid values; the resulting per-entry cost is $O(1)$ for the
bias-corrected weighted formulation versus $O(m_i)$ for the product.
\item Regulator-specific thresholds must be retained (in contrast
with the Samuilik formulation \eqref{eq:samuilik_def}, which collapses
them into a shared $\theta_i^{\mathrm{S}}$).
\item The dynamical regime of interest is near the critical point
(half-maximal regulatory configurations), where the bias correction
forces zeroth-order matching (exact value $1/2^{m_i}$) with the
product. The first-order discrepancy is linear in the displacement
with coefficient $(1-2^{1-m_i})\,\partial_l\fprod$ derived from the
slope-ratio formula (\Cref{prop:slope_ratio}), and the second-order
curvature mismatch is $\lambda^2(2^{m_i}-1)(2^{m_i}-2)/8^{m_i}$
(\Cref{prop:curvature_mismatch}).
\item The parameter set lies on (or near) the self-consistency
sub-manifold of \Cref{thm:shared_equilibrium}, where the two
formulations share a canonical critical-point equilibrium and the
linearised dynamics are related by the explicit spectral formula
\eqref{eq:eigenvalue_map}. In this regime the bias-corrected
weighted formulation is a strict simplification of the
product-of-logistics with no loss of equilibrium structure (modulo the
spectral re-scaling and the curvature mismatch).
\item Numerical integration efficiency is a concern. Replacing
unbounded linear-additive regulation by the bounded
bias-corrected-weighted form substantially reduces the global
Lipschitz constants of both the vector field and its Jacobian; on the
Vinoth two-gene network, Belgacem~\cite{belgacem2026beyond} reports a
$\sim 70\%$ reduction in the vector-field constant and a $\sim 65\%$
reduction in the Jacobian Lipschitz constant, permitting
correspondingly larger explicit-integrator step sizes.
\end{itemize}

\paragraph*{Samuilik weighted-sum \eqref{eq:samuilik_def}.}
The Samuilik formulation is preferable only when network sparsity
combined with real-valued continuously graded weights is essential and
the negative-midpoint repression pathology
(\Cref{subsec:repression_pathology}) is biologically tolerable. We
note that the formulation has been productively used in qualitative
network analyses where the discrepancy with the
product-of-logistics is not the analytical
focus~\cite{samuilik2022mathematical,kozlovska2023quasi}, but its
quantitative use as a substitute for the product (or for the
bias-corrected weighted formulation) is structurally limited by the
threshold and repression issues established in
\Cref{sec:samuilik_comparison}.

\subsection{Implications for control synthesis}
\label{subsec:control_implications}

The compactness of the bias-corrected weighted formulation is
particularly attractive for control-theoretic applications. The
single-sigmoid Jacobian below is a structurally simpler object to
plug into the qualitative-control framework
of~\cite{chambon2020qualitative}, in which undesired oscillations in
a genetic negative-feedback loop are suppressed by a piecewise-affine
controller with uncertain measurements, and into the
negative-feedback-loop control architecture
of~\cite{belgacem2021control}. The Jacobian of \eqref{eq:bcw_model}
at $\bx$ is
\begin{equation}
\frac{\partial\fwgtbc}{\partial x_l}(\bx)
\;=\; w_{il}\,\lambda\,\sigma'\!\bigl(\lambda S_i(\bx) -
                          \log(2^{m_i}-1)\bigr),
\label{eq:bcw_jacobian}
\end{equation}
a single self-referential expression evaluated once per gene per
state, in contrast with the product-of-logistics Jacobian
\begin{equation}
\frac{\partial\fprod}{\partial x_l}(\bx)
\;=\; w_{il}\,\lambda\,\sigma'(u_l)\!\!\prod_{j\neq l}\!\sigma(u_j),
\end{equation}
which requires $m_i$ sigmoid evaluations and a product. For
high-dimensional networks ($n\gg 100$) with $m_i\gg 1$ regulators per
gene, this difference can yield substantial computational savings in
gradient-based optimisation (e.g.\ inverse-problem parameter
estimation, gradient-flow control), provided the dynamical regime of
interest is consistent with the critical-point matching of
\Cref{thm:three_point}.

\subsection{Outlook on high-dimensional systems}
\label{subsec:high_dim}

Biological networks at the cellular scale integrate hundreds or
thousands of genes coupled through enzymatic and metabolic feedbacks.
For such large $n$, the analytical advantages of logistic functions
over Hill functions---global $C^\infty$-smoothness, closed-form
self-referential derivatives, uniformly bounded Lipschitz constants,
elementary inverses via the logit transformation---become
quantitatively important. Computer-aided high-dimensional analyses of
Glass networks~\cite{belgacem2025glass} have already exposed the
combinatorial richness of periodicity, chaos, and bifurcation
structure in such large-scale circuits, and the bias-corrected
weighted-logistic formulation provides a smooth analytic surrogate in
which the same combinatorial structure is preserved while the
analytical complexity collapses to a single sigmoid per gene.
The bias-corrected weighted formulation
extends the logistic-framework advantages to the
single-sigmoid-per-gene regime, with
two complementary roles. \textbf{First}, as a substitute for the
product-of-logistics in control-design pipelines where the
multiplicative Jacobian structure is computationally prohibitive but
near-critical-point operation is the design point. \textbf{Second}, as
a building block for systematic linearisation schemes whose
reduced-order models inherit the three-point matching at the network
level. Quantifying these advantages on realistic genome-scale networks,
and identifying regimes where bias-corrected weighted formulations
preserve qualitative dynamical features (multistability, oscillations,
extreme events) of the underlying product-of-logistics, are natural
directions for future work. Hybrid-systems extensions---in which the
sigmoid steepness is allowed to diverge, recovering Glass-network
switches at the price of Zeno-like behaviour---may benefit from the
probabilistic-convolution regularisation
of~\cite{belgacem2019zeno}.

\section{Conclusion}
\label{sec:conclusion}

We introduced a bias-corrected weighted-logistic model for gene
regulatory networks (\Cref{def:bcw_model}) that bridges the
analytical compactness of single-sigmoid weighted-sum formulations
with the biological interpretability of product-of-logistics models.
The bias term $b_i = -\lambda^{-1}\log(2^{m_i}-1)$ depends only on the
combinatorial structure of the network through the regulator count
$m_i$, and a canonicity result (\Cref{prop:bias_canonicity}) shows it
is the \emph{unique} single-sigmoid correction recovering the product's
critical-point value $1/2^{m_i}$ and sharing its canonical equilibrium
under self-consistency. The bias-corrected weighted regulatory function
coincides with the product-of-logistics at three biologically
significant reference configurations: the critical point, the
low-steepness limit, and the saturation limits
(\Cref{thm:three_point}).

The bias-corrected weighted-logistic ODE inherits the global
analytical guarantees of the product-of-logistics framework
of~\cite{belgacem2025exploring,belgacem2026logistic}: global
existence of $C^\infty$ solutions, an explicit Lipschitz constant
$L_F\le M = \max_i(\kappa_i\,\lambda m_i/4 + \gamma_i)$, a positively
invariant hyper-rectangle $\prod_i[0,\kappa_i/\gamma_i]$
(\Cref{thm:well_posedness}), and strictly-positive basal output
$\fwgtbc(\bx) > 0$ preventing the off-state trap
(\Cref{cor:no_shutdown}). These properties make the bias-corrected
weighted-logistic formulation a structurally well-posed
single-sigmoid surrogate for the product framework, free of the
Hill-induced pathologies catalogued
in~\cite{belgacem2025exploring,belgacem2026logistic}.

The discrepancy between the two formulations admits the closed-form
algebraic identity \eqref{eq:algebraic_identity}, expressing
$\fprod-\fwgtbc$ as a sum over the proper non-empty subsets of the
regulator index set (\Cref{thm:algebraic_identity}), with sharp
leading-order asymptotic decay rates in both saturation regimes
(\Cref{prop:asymptotic_rates}). The local sensitivity ratio at the
critical point is exactly $2-2^{1-m_i}$ (\Cref{prop:slope_ratio}),
monotonically approaching $2$ as $m_i\to\infty$. The structural
symmetry $\sigma''(\theta)=0$ that underlies closed-form supercritical
bifurcation theorems for the product is broken by the bias correction
for $m_i\ge 2$ (\Cref{prop:curvature_mismatch}). Three complementary
structural results characterise the bias as canonical: no
state-independent constant shift can yield global equivalence
(\Cref{prop:no_global_eq}); the bias is the unique zero-order
mean-field approximation of the log-sum-exp correction at the critical
point (\Cref{prop:mean_field}); and no affine single-sigmoid family
can simultaneously match the product's value, slope, and curvature at
the critical point (\Cref{prop:affine_limit}).

A self-consistency condition $\theta_{ij}=(\kappa_j/\gamma_j)/2^{m_j}$
on the parameters singles out a parameter sub-manifold on which the
product-of-logistics and bias-corrected weighted dynamical systems
share a canonical critical-point equilibrium with regulatory output
$1/2^{m_i}$ (\Cref{thm:shared_equilibrium}); in the contractive
regime $\gamma_i > \kappa_i\lambda m_i/4$, this shared equilibrium is
\emph{globally unique and exponentially attracting} for both systems
(\Cref{cor:global_unique}), and any threshold perturbation
$\boldsymbol\eta$ from self-consistency propagates linearly to
equilibrium deviation through a parameter-Jacobian matrix that itself
obeys the row-wise rescaling $\Psi^{\rm bcw} = D\Psi^{\rm prod}$
(\Cref{prop:robustness_eq}). At a shared equilibrium, the dynamical
Jacobians of the two systems are related by the explicit transformation
$J^{\rm bcw} = D(J^{\rm prod}+\Gamma)-\Gamma$
(\Cref{prop:linearisation}), and the bias correction satisfies a
stability dichotomy (\Cref{cor:stability_dichotomy}) in which modes
of $J^{\rm prod}$ with $\operatorname{Re}\mu^{\rm prod} > -\gamma$ are
destabilised, modes with $\operatorname{Re}\mu^{\rm prod} < -\gamma$
are further stabilised, and the threshold modes with
$\operatorname{Re}\mu^{\rm prod} = -\gamma$ are preserved exactly.

The framework is structurally distinct from the Samuilik weighted-sum
formulation \eqref{eq:samuilik_def}, which loses regulator-specific
thresholds and exhibits a negative-midpoint repression pathology
(\Cref{subsec:repression_pathology}). Within the product-of-logistics
framework, real-valued positive weights are equivalent to fixed unit
weights under parameter rescaling
(\Cref{prop:weight_rescaling}), an identifiability constraint that
does not extend to the Samuilik formulation.

A two-gene negative-feedback oscillator with self-activation
illustrates these results concretely (\Cref{sec:application}). Both
the product-of-logistics and the bias-corrected weighted formulation
share the canonical critical-point equilibrium $\bx^* =
(\kappa_A/(4\gamma),\,\kappa_B/(4\gamma))$ with two distinct biological
scales $\theta_A\ne\theta_B$. The bias-corrected formulation
undergoes Hopf bifurcation at $\lambda^{\rm bcw}_{\rm Hopf} =
32\gamma/(3(\kappa_A+\kappa_B))$, strictly smaller than the
product's threshold $\lambda^{\rm prod}_{\rm Hopf} = 16\gamma/
(\kappa_A+\kappa_B)$ by the factor $2/3$ predicted by the
eigenvalue map of \Cref{prop:linearisation}, and exhibits
sustained oscillations around $\bx^*$ in the intermediate regime
$\lambda\in(\lambda^{\rm bcw}_{\rm Hopf},\,\lambda^{\rm prod}_{\rm Hopf})$,
where the product is linearly stable (\Cref{fig:dynamics}). The
bias correction is therefore not a recalibration of equilibria but
a structural shift in the bifurcation diagram.

Future work will pursue: (i) extensions of the bias correction to
heterogeneous steepness and to weighted aggregation rules beyond
unit-magnitude; (ii) systematic linearisation and model-reduction
schemes for high-dimensional bias-corrected weighted networks
preserving stability and bifurcation structure; (iii) parameter
estimation from experimental time-series data, exploiting the
modular identifiability of regulator-specific thresholds; and
(iv) applications to control-theoretic synthesis, observer design,
and stochastic and spatiotemporal extensions of the framework.


\end{document}